\documentclass[a4paper,11pt]{article}
\usepackage[english]{babel}
\usepackage[utf8]{inputenc}
\usepackage[T1]{fontenc}

\usepackage{authblk}

\usepackage{newunicodechar}
\newunicodechar{ζ}{$\zeta$}

\usepackage{hyperref}
\usepackage{graphicx}

\usepackage{amsmath,amssymb,amsthm,amsbsy}
\usepackage{bbm} 

\usepackage{enumitem} 

\usepackage{geometry}
\newtheorem*{remark*}{Remark}

\newtheorem{theorem}{Theorem}
\newtheorem{proposition}[theorem]{Proposition}
\newtheorem{corollary}[theorem]{Corollary}

\newcommand{\Z}{\mathbb{Z}} 
\newcommand{\C}{\mathbb{C}} 
\newcommand{\N}{\mathbb{N}} 
\newcommand{\R}{\mathbb{R}} 

\newcommand{\bu}{\mathbf{u}}

\newcommand{\Det}[1]{\mathrm{det}\left(#1\right)}

\newcommand{\calS}{\mathcal{S}}

\providecommand{\keywords}[1]{%
  \par\medskip\noindent\textbf{Keywords:} #1\par\medskip}

\newcommand{\proscal}[2]{\left<#1,#2\right>}

\newcommand{\set}[1]{\left\{#1\right\}}

\newcommand{\eqdef}{\overset{\mathrm{def}}{=}}

\newcommand{\cKz}{\mathcal{K}^{{(0)}}} 
\newcommand{\cKh}{\mathcal{K}^{{(1/2)}}} 
\newcommand{\cKnu}{\mathcal{K}^{{(\nu)}}} 

\newcommand{\Cl}{\mathrm{C}} 
\newcommand{\Sl}{\mathrm{S}} 

\newcommand{\Clh}{\mathrm{C}^{(1/2)}} 
\newcommand{\Slh}{\mathrm{S}^{(1/2)}} 
\newcommand{\Slnu}{\mathrm{S}^{(\nu)}} 

\newcommand{\Li}{\mathrm{Li}} 
\newcommand{\Liz}{\mathrm{Li}^{{(0)}}} 
\newcommand{\Lih}{\mathrm{Li}^{{(1/2)}}} 
\newcommand{\Linu}{\mathrm{Li}^{{(\nu)}}} 

\newcommand{\Chi}{\chi}

\newcommand{\Chiz}{\chi^{(0)}}
\newcommand{\Chih}{\chi^{(1/2)}}

\newcommand{\RE}{\mathrm{Re}}

\newcommand{\stirlingfirst}[2]{\genfrac{[}{]}{0pt}{}{#1}{#2}}

\begin{document}
\title{From Inverse Vandermonde Matrices
       to the Shift Structure of the Dirichlet Eta and Beta Functions
      }
       
\author{Serge Iovleff\thanks{serge.iovleff@utbm.fr}}

\affil{Université de technologie de Belfort Montbéliard\\
SINERGIES (UR 4662),\\
F-90000 Belfort, France
}

\maketitle
\begin{abstract}
We investigate the asymptotic behavior of the inverse of Vandermonde
matrices with nodes drawn from the quadratic sequences \(u_n=n^2\) and
\(u_n=(n-\tfrac12)^2\).  We derive new limiting formulas for the entries
of the inverses, expressed through the even and odd central factorial
numbers.  These formulas are then used to study the asymptotic behavior
of a ratio of determinants involving a modified Vandermonde matrix.  For
each node sequence, the corresponding determinant ratio converges
uniformly on compact subsets of the complex plane to an explicit
expression involving the Dirichlet eta or beta function.

We then study a family of related functions and, using an inversion
theorem, exhibit their shift structure: the shifted values of the eta and
beta functions are recovered as finite linear combinations of the family,
which yields an integral form of these values as integrals of
polynomials related to the Bernoulli and Euler polynomials against
generalized (half-)Clausen functions.  A unified form of the shift representation,
valid for an arbitrary base
point \(z\) with \(\RE(z)>0\), brings the eta and beta families under a
single theorem.  On the critical line the representation is unitary and
the values on the two sides are complex conjugates.
\end{abstract}


\section*{Foreword}

I would like to advise the reader that this article is the result of a
stroll through different themes and unexpected encounters.

The stroll started with the study of a ratio of determinants
(Section~\ref{sec:ratio-determinant}).  In order to analyze its
asymptotic behavior, I was led to study the asymptotics of the inverse
of a class of Vandermonde matrices (Section~\ref{sec:InverseVandermonde}).
In this step, I discovered a new way to study the asymptotics of the
coefficients of the inverse Vandermonde matrix.  I then sought to deepen
the connections between certain power series and Dirichlet functions
appearing in Section~\ref{sec:ratio-determinant}, which led me to
discover a shift representation of the eta and beta functions in terms of
polynomials related to the Bernoulli and Euler polynomials, together with
the integral form of this representation against a family of kernels
(Section~\ref{sec:shift-structure}).

While I was working on the emergence of these results, I crossed paths
with fellow travelers called LLMs or generative AI.  It was six months
ago -- we are in October 2026 -- and what had been a stroll suddenly
turned into an exploratory race, with my new companions taking on the
task of testing the various hypotheses I could express and putting them
into form.  In particular, the last section
(Section~\ref{sec:unified-kernel}), which allows all the shifts of the
eta and beta functions to be expressed in a unified way, was co-created
with these new companions.  See the Afterword at the end of this paper
for the disclaimer.

This paper is the result of this wandering and these encounters.

\section{Introduction}
\label{sec:introduction}

Let \(0<u_1<u_2<\ldots<u_N<\ldots\) be an increasing sequence of positive real
numbers. For \(N>1\), consider the Vandermonde matrix
\begin{equation}\label{eq:VN-Def}
V_N(\bu) \eqdef
\begin{pmatrix}
  1 & u_1    & u_1^2  & \cdots & u_1^{N-1} \\
  1 & u_2    & u_2^2  & \cdots & u_2^{N-1} \\
\vdots       & \vdots & \ddots & \vdots & \vdots \\
  1 & u_N    & u_N^2  & \cdots & u_N^{N-1} \\
\end{pmatrix},
\end{equation}
with \(\bu = (u_1,\dots,u_N)\). The entries of \(W_N(\bu) = V_N(\bu)^{-1}\)
admit explicit expressions in terms of elementary symmetric polynomials, as
recalled in Appendix~\ref{app:inverseVandermonde} and
equation~\eqref{eq:WknN}.

In Section~\ref{sec:InverseVandermonde}, we investigate the asymptotic
behavior of the entries of \(W_N(\bu)\) as \(N \to \infty\) for two
distinguished choices of the nodes \(u_n\): the squares \(u_n = n^2\) and the
shifted squares \(u_n = (n-\tfrac12)^2\). These choices are natural because
the elementary symmetric polynomials appearing in the general
formula~\eqref{eq:WknN} specialize to well-known families of combinatorial
numbers. These connections are exhibited in
Subsection~\ref{subsec:special-cases} and summarized in
Table~\ref{tab:specializations}.
In Subsection~\ref{subsec:asymptotic} we take the limit
in \(N\) and find that the asymptotic coefficients \(w_{kn}^{[\infty]}\)
for fixed row and column indices $k$ and $n$,
admit integral representations and closed-form sums involving elementary
constants and powers of \(\pi\). The limiting forms are given in
Tables~\ref{tab:wkn-representations} and~\ref{tab:w1n-values}.
While explicit formulas
for the entries of the inverse Vandermonde matrix are classical (see
\cite{knuth1968art, Macon1958, Turner1966}), the asymptotics of
the individual entries as \(N\rightarrow\infty\) have not been investigated
and the results are new.

In Section~\ref{sec:ratio-determinant}, we study the convergence of
a ratio of determinants introduced by us in an unpublished
paper~\cite{iovleff2022}. Let \(s\in\C\), and consider the modified
Vandermonde matrix
\begin{equation*}
V^{(s)}_N(\bu) \eqdef
\begin{pmatrix}
  u_1^{-s} & u_1    & u_1^2  & \cdots & u_1^{N-1} \\
  u_2^{-s} & u_2    & u_2^2  & \cdots & u_2^{N-1} \\
\vdots     & \vdots & \ddots & \vdots & \vdots \\
  u_N^{-s} & u_N    & u_N^2  & \cdots & u_N^{N-1} \\
\end{pmatrix}.
\end{equation*}
If \(s=0\), the resulting matrix is the Vandermonde matrix \(V_N(\bu)\)
given in equation~\eqref{eq:VN-Def}. We define the ratio
\(R_N(s;\bu)\) as
\begin{equation}\label{eq:ratio}
R_N(s;\bu) \eqdef \frac{\Det{V^{(s)}_N(\bu)} }{ \Det{V_N(\bu)} }.
\end{equation}
Using the results we found in Section~\ref{sec:InverseVandermonde},
and applying a theorem of Biane, Pitman, and Yor~\cite{biane2001probability},
we obtain the uniform convergence of this ratio on compact subsets of the complex
plane for the two node sequences. The limit is then
identified with the classical Dirichlet eta and beta functions, up to explicit
prefactors. The results are collected in Table~\ref{tab:S-infty}.

The final two sections are devoted to the study of the limiting functions
\[
\alpha^{[\infty]}_k(s)\eqdef\sum_{n=1}^\infty w_{kn}^{[\infty]}\frac{1}{u_n^s},
\]
where \(w_{kn}^{[\infty]}\) denotes the limit as \(N\to\infty\) of the
entry in row \(k\), column \(n\), of the inverse Vandermonde matrix
\(W_N(\bu)\).  These functions form a family indexed by \(k\).  The first
member is, up to an explicit normalization, the eta function \(\eta(2s)\)
when \(u_n=n^2\), and the beta function \(\beta(2s+1)\) when
\(u_n=(n-\tfrac12)^2\).  The family satisfies a recurrence that expresses
each member as a finite linear combination of the first member at
shifted values of \(s\), with coefficients given by powers of \(\pi\).

In Section~\ref{sec:shift-structure} we invert this recurrence: the first
member, at any shifted value of \(s\), is recovered as a finite linear
combination of the members at the base value of \(s\).  Substituting the
integral forms of the \(\alpha^{[\infty]}_k\)'s
(Theorem~\ref{th:Wkn-infty-n2}) then gives a shift representation, in
which the shifted values \(\eta(2s+2n)\) and \(\beta(2s+2n+1)\) are
expressed as integrals of the generalized (half-)Clausen functions
against polynomials related to the Bernoulli and Euler polynomials.
These formulas, collected in Corollary~\ref{cor:eta-beta-clausen},
constitute the integral form of the shift representation and hold for
every \(s\) with \(\RE(s)\ge\tfrac12\).

Section~\ref{sec:unified-kernel} unifies and extends these results.
Using the analytic continuation of the kernels
(Theorem~\ref{thm:continuation}), we show that the shift representation
is not restricted to the discrete arguments of
Corollary~\ref{cor:eta-beta-clausen}: for every \(z\) with \(\RE(z)>0\),
every even shifted value of \(\eta\) and \(\beta\) to the right of \(z\) is
obtained from the kernel \(\Linu_{z-1}\) and a constant independent
of \(z\) (Theorem~\ref{th:unified-kernels}).  Combined with the
functional equation, this generates all even integer shifts of both functions
from the conjugate base points $z$ and $1-z$ in the strip \(0<\RE(z)<1\). 
The boundary case \(z=1\) is analyzed separately as it is the only place
where the representation fails.

\section{Special Vandermonde Matrices and Their Asymptotic Inverses}
\label{sec:InverseVandermonde}

The results of this section follow from an elementary recurrence
relation satisfied by the entries of the matrices
$(W_1(\bu),\ldots W_N(\bu))$ which does not seem to have been noticed
before. Denoting by \(\left(w_{kn}^{[N]}\right)_{k,n=1}^N\) the
entries of the matrix \(W_{N}(\bu)\) and by
\(\calS_{\bu}=\{u_{1},u_{2},\ldots,u_{N}\}\) the set of indeterminates,
it is established in Appendix~\ref{app:inverseVandermonde} that
\begin{equation*}
 w_{kn}^{[N]} = (-1)^{N-k} \frac
{\sigma_{N-k}^{N-1}(\calS_\bu\backslash \left\{ u_n \right\})}
{\displaystyle\prod_{\substack {1 \leq j \leq N \\ j \neq n}} (u_n - u_j)},
\quad 1\leq k,n \leq N
\end{equation*}
with \(\sigma_{N-k}^{N-1}\) denoting the elementary symmetric polynomials of
degree \(N-k\) in \(N-1\) variables (see \cite{macdonald1998symmetric})
and the convention \(\sigma_0^{N-1}=1\). Assuming \(N\geq2\) and
\(1\leq k,n \leq N-1\), and using the recurrence definition of the elementary
symmetric polynomials, we obtain
\begin{align*}
\sigma^{N-1}_{N-k}(\calS_{\bu}\backslash \{u_n\})
& = \sigma^{N-2}_{N-k}(\calS_{\bu}\backslash \{u_n,u_N\} )
  + u_N \sigma^{N-2}_{N-k-1}(\calS_{\bu}\backslash \{u_n,u_N\} )\\
& =  \sigma^{N-2}_{N-k}(\calS_{\bu}\backslash \{u_n,u_N\} )
  + u_n \sigma^{N-2}_{N-k-1}(\calS_{\bu}\backslash \{u_n,u_N\} )
  - (u_n - u_N) \sigma^{N-2}_{N-k-1}(\calS_{\bu}\backslash \{u_n,u_N\} )\\
& = \sigma^{N-1}_{N-k}(\calS_{\bu}\backslash \{u_N\})
  - (u_n - u_N) \sigma^{N-2}_{N-k-1}(\calS_{\bu}\backslash \{u_n,u_N\} ).
\end{align*}
Thus, when \(1\leq k,n \leq N-1\), the coefficients of the matrices
\((W_{N})_{N\geq 1}\) follow the recurrence relation
\[
w_{kn}^{[N]} =
(-1)^{N-k} \frac{\sigma_{N-k}^{N-1}(\calS_\bu\backslash \left\{ u_N \right\})}
{\displaystyle\prod_{\substack {1 \leq j \leq N \\ j \neq n}} (u_n - u_j)}
+ w_{kn}^{[N-1]}.
\]
Expanding the recursion, we obtain the general expression
\begin{equation}
\label{eq:WknN}
w_{kn}^{[N]} = \sum_{l=k \vee n}^N (-1)^{l-k}
\frac{\sigma_{l-k}^{l-1}(\set{u_1,\ldots,u_{l-1}})}
{\displaystyle\prod_{\substack {1 \leq j \leq l \\ j \neq n}} (u_n - u_j)},  \qquad 1\le k,n \le N,
\end{equation}
with $k \vee n$ denoting $\max(k,n)$.
\subsection{Special Cases}
\label{subsec:special-cases}
We consider two node sequences for which elementary
symmetric polynomials specialize to known families of numbers: \(u_n=n^2\) and \(u_n=(n-\tfrac12)^2\), which correspond to the even and odd unsigned central
factorial numbers, respectively. These sequences are described in Appendix~\ref{app:cfn}.
The results of the computation are summarized
in Table~\ref{tab:specializations} with the following conventions
that will be used throughout this section:
\begin{itemize}
\item  \(t_0(k,l) = |t(2l,2k)| \) are the even unsigned central factorial numbers,
\item \(t_{\frac12}(k,l) = |t(2l+1,2k+1)|\) are the odd unsigned central
factorial numbers.
\end{itemize}

\begin{table}[htb!]
\centering\small
\[
\begin{array}{c|c|l|l|}
u_n & \sigma_{l-k}^{l-1}({u_1,\ldots,u_{l-1}})
& \multicolumn{1}{c|}{\displaystyle\prod_{\substack{1 \le j \le l \\ j \ne n}} (u_n-u_j)}
& \multicolumn{1}{c|}{w_{kn}^{[N]}} \\
\hline
n^2 &  
t_0(k,l) & (-1)^{l-n} \dfrac{(l-n)! (l+n)!}{2n^2} 
& (-1)^{k+n}2 n^2 \sum\limits_{l=k \vee n}^{N} \binom{2l}{l-n} \dfrac{t_0(k,l)}{(2l)!} \\
\left(n-\tfrac12\right)^2 & 
t_{\frac12}(k-1,l-1) & (-1)^{l-n} \dfrac{(l-n)! (n+l-1)!}{2n-1} & (-1)^{k+n}
 (2n-1) \sum\limits_{l=k \vee n}^{N} \binom{2l-1}{l-n} \dfrac{t_{\frac12}(k-1,l-1)}{(2l-1)!}
\end{array}
\]
\caption{Closed-form expressions for 
the elementary symmetric polynomials \(\sigma_{l-k}^{l-1}\), the restricted
product factors, and the resulting coefficients \(w_{kn}^{[N]}\) for
the two quadratic choices of the numbers \(u_n\).
}
\label{tab:specializations}
\end{table}

\subsection{Asymptotic}
\label{subsec:asymptotic}

The aim of this section is to provide an expression for the coefficients
\(w_{kn}^{[\infty]} \eqdef \lim_{N\rightarrow\infty} w_{kn}^{[N]}\),
for fixed row and column indices $k$ and $n$.
To prove the main results, we will make use of the following integral
representation of the binomial coefficients, which can be found in 
\cite{Egorychev:1984:IRC}:
\begin{equation}\label{eq:binom-integral}
\binom{m}{n} 
= \frac{1}{2\pi} \int_{-\pi}^{\pi} e^{-i n\theta}\left(1 + e^{i\theta}\right)^m \, d\theta
= \frac{1}{2\pi} \int_{-\pi}^{\pi} e^{-i\left(n-\frac{m}{2} \right)\theta}
\left(2\cos \frac{\theta}{2}\right)^m \, d\theta.    
\end{equation}
This identity follows by expanding \(\left(1 + e^{i\theta}\right)^m\) via the 
binomial theorem and applying the orthogonality relation
\[
\frac{1}{2\pi} \int_{-\pi}^{\pi} e^{-i m\theta} e^{i k\theta} \, d\theta 
= \delta_{mk},
\]
which selects precisely the term \(k = m\) in the expansion.
We now state the main theorem of this section.
\begin{theorem}\label{th:Wkn-infty-n2}
When \(u_n=n^2\), we have
\begin{equation}\label{eq:Wkn-infty-0}
w_{kn}^{[\infty]} = (-1)^{k+1} \frac{2n}{(2k-1)!} \frac{1}{\pi} \int_0^{\pi} 
\theta^{2k-1} \sin(n\theta)\,d\theta,
\end{equation}
and when \(u_n = \left(n-\tfrac12\right)^2\), we have
\begin{equation}\label{eq:Wkn-infty-12}
w_{kn}^{[\infty]} =
(-1)^{k+1} \frac{2(n-\tfrac12)}{(2k-1)!} \frac{1}{\pi} \int_0^{\pi} 
\theta^{2k-1} \sin\left((n-\tfrac12)\theta\right) \, d\theta.
\end{equation}
\end{theorem}

\begin{proof}
We first prove \eqref{eq:Wkn-infty-0}. Replacing the binomial coefficient 
by its integral representation \eqref{eq:binom-integral}, we obtain
\[
(-1)^{k+n}2 n^2 \sum_{l=k \vee n}^{\infty} \binom{2l}{l-n} \frac{t_0(k,l)}{(2l)!}
= \frac{(-1)^{k+n}2 n^2}{\pi} \sum_{l=k \vee n}^{\infty} \int_{0}^{\pi} 
\cos(n\theta) \left(2\cos\frac{\theta}{2}\right)^{2l}
\frac{t_0(k,l)}{(2l)!} \, d\theta.
\]
For \(l<n\), the function \(\left(2\cos\frac{\theta}{2}\right)^{2l}\) is a 
trigonometric polynomial of degree \(l\) in \(\theta\); it is therefore 
orthogonal to \(\cos(n\theta)\) on \([0,\pi]\), so the corresponding integrals vanish.
Since \(t_0(k,l)=0\) for \(l<k\), the summation may be started at \(l=k\).

As \(t_0(k,l)\ge0\) and \(2\cos(\theta/2)\le2\) for all \(\theta\in[0,\pi]\),
every partial sum of the series is dominated,
uniformly in \(N\) and \(\theta\), by the value at \(z=2\) of the
generating function \eqref{eq:gen-funct-cfn-0} given in Appendix~\ref{app:cfn}:
\[
\sum_{l=k}^{N}\frac{t_0(k,l)}{(2l)!}\,2^{2l}
\le \sum_{l=k}^{\infty}\frac{t_0(k,l)}{(2l)!}\,2^{2l}
= \frac{1}{(2k)!} \left[2\arcsin(1)\right]^{2k} = \frac{\pi^{2k}}{(2k)!}.
\]
The interchange of summation and
integration is therefore justified by dominated convergence.
Elementary trigonometry shows that
\[
2\arcsin\left(\cos \frac{\theta}{2} \right) = \pi-\theta
\]
for \(\theta \in [0,\pi]\). Therefore,
\[
w_{kn}^{[\infty]} =
\frac{(-1)^{n+k}}{\pi} \frac{2n^2}{(2k)!} \int_{0}^{\pi} 
\left(\pi-\theta \right)^{2k} \cos(n\theta) \, d\theta.
\]
The change of variable \(\phi=\pi-\theta\) gives
\[
\int_{0}^{\pi}(\pi-\theta)^{2k}\cos(n\theta)\,d\theta
= \int_{0}^{\pi}\phi^{2k}\cos\left(n(\pi-\phi)\right)\,d\phi
= (-1)^{n}\int_{0}^{\pi}\phi^{2k}\cos(n\phi)\,d\phi,
\]
so that
\[
w_{kn}^{[\infty]} =
\frac{(-1)^{k}}{\pi} \frac{2n^2}{(2k)!} \int_{0}^{\pi} 
\theta^{2k} \cos(n\theta) \, d\theta.
\]
Integrating by parts with \(u=\theta^{2k}\) and \(dv=\cos(n\theta)\,d\theta\),
and noting that the boundary terms vanish because \(\sin(n\pi)=\sin(0)=0\),
yields
\[
\int_{0}^{\pi}\theta^{2k}\cos(n\theta)\,d\theta
= -\frac{2k}{n}\int_{0}^{\pi}\theta^{2k-1}\sin(n\theta)\,d\theta.
\]
Substituting this expression back completes the proof of the first part.

For the case \(u_n = (n-\tfrac12)^2\), the same reasoning leads to
\[
w_{kn}^{[\infty]} = (-1)^{k+n} (2n-1) 
\frac{1}{\pi} \int_0^{\pi} \cos\left((n-\tfrac12)\theta\right)
\sum_{l=k}^{\infty} \frac{t_{\frac12}(k-1,l-1)}{(2l-1)!} 
\left(2\cos\frac{\theta}{2}\right)^{2l-1} \, d\theta.
\]
Substituting \(z = 2\cos\frac{\theta}{2}\) in the generating function 
\eqref{eq:gen-funct-cfn-12} and again using 
\(2\arcsin(\cos\frac{\theta}{2}) = \pi-\theta\) for \(\theta\in[0,\pi]\), 
we obtain
\[
w_{kn}^{[\infty]} = (-1)^{k+n} \frac{2n-1}{(2k-1)!} \frac{1}{\pi}
\int_0^{\pi} (\pi-\theta)^{2k-1} \cos\left((n-\tfrac12)\theta\right) \, d\theta.
\]
Performing the change of variable \(\phi=\pi-\theta\) and using
\(\cos\left((n-\tfrac12)(\pi-\phi)\right) 
= (-1)^{n-1} \sin\left((n-\tfrac12)\phi\right)\) completes the proof.
\end{proof}

\subsection{Consequences}
\label{subsec:consequences}
Integrating by parts iteratively, we obtain the following closed-form 
expressions for the asymptotic coefficients \(w_{kn}^{[\infty]}\).

\begin{corollary}\label{cor:Wkn-closed-forms}
For all \(k \ge 1\) and \(n \ge 1\), the coefficients 
\(w_{kn}^{[\infty]}\) admit the following explicit sums.

\begin{enumerate}[label=(\alph*)]
\item When \(u_n = n^2\), we have
\begin{equation*}
w_{kn}^{[\infty]} = \frac{2(-1)^{n}}{\pi^2} 
\sum_{l=1}^{k} (-1)^{l}\,\frac{\pi^{2l}}{(2l-1)!\,n^{2k-2l}}.
\end{equation*}

\item When \(u_n = (n-\tfrac12)^2\), we have
\begin{equation*}
w_{kn}^{[\infty]} = \frac{2(-1)^{n}}{\pi}\,
\sum_{l=1}^{k} (-1)^{l}\,
\frac{\pi^{2l-2}}{(2l-2)!\,\left(n-\tfrac12\right)^{2k-2l+1}}.
\end{equation*}
\end{enumerate}
\end{corollary}
\begin{proof}
The proof follows by repeated integration by parts applied to the 
integral representations established in Theorem~\ref{th:Wkn-infty-n2}.
Each integration by parts reduces the power of $\theta$ by one and
produces a boundary term that either vanishes or simplifies to a known
value  using the elementary identities \(\cos(n\pi) = (-1)^n\),
\(\sin(n\pi) = 0\), and  \(\sin\left((n-\tfrac12)\pi\right) = (-1)^{n-1}\).
The details are straightforward and left to the reader.
\end{proof}
The integral representations of Theorem~\ref{th:Wkn-infty-n2} and the
closed-form sums of Corollary~\ref{cor:Wkn-closed-forms} are summarized in
Table~\ref{tab:wkn-representations} below.
\begin{table}[htb!]
\centering
\[
\begin{array}{c|c|c|}
u_n 
& w_{kn}^{[\infty]}\text{ (sum)} 
& w_{kn}^{[\infty]}\text{ (integral)} \\
\hline
n^2 
& \dfrac{2(-1)^{n}}{\pi^2}\displaystyle\sum_{l=1}^{k}\dfrac{(-1)^{l}\pi^{2l}}{(2l-1)!\,n^{2(k-l)}} 
& (-1)^{k+1}\dfrac{2n}{(2k-1)!\pi}\displaystyle\int_0^\pi\!\theta^{2k-1}\sin(n\theta)\,d\theta \\[16pt]
(n-\tfrac12)^2 
& \dfrac{2(-1)^{n}}{\pi}\displaystyle\sum_{l=1}^{k}\dfrac{(-1)^{l}\pi^{2l-2}}{(2l-2)!\,(n-\tfrac12)^{2k-2l+1}} 
& (-1)^{k+1}\dfrac{2(n-\tfrac12)}{(2k-1)!\pi}\displaystyle\int_0^\pi\!\theta^{2k-1}\sin\left((n-\tfrac12)\theta\right)d\theta
\end{array}
\]
\caption{Representations of the asymptotic coefficients
\(w_{kn}^{[\infty]}\) for the two node sequences: the closed-form sum
(Corollary~\ref{cor:Wkn-closed-forms}) and the integral form
(Theorem~\ref{th:Wkn-infty-n2}).}
\label{tab:wkn-representations}
\end{table}

We now specialize each case to \(k = 1\), as these values will
be needed  in the next section. The results are summarized in
Table~\ref{tab:w1n-values}.
Observe that the limiting values for \(k=1\) can be obtained directly
from the finite \(N\)-formulas via Stirling's formula, without passing
through the integral representations of Theorem~\ref{th:Wkn-infty-n2}.
For \(u_n = n^2\), Stirling's formula gives \(\binom{2N}{N+n}/\binom{2N}{N} \to 1\)
as \(N \to \infty\) for fixed \(n\). For \(u_n = (n-\tfrac12)^2\), Stirling's
formula gives for fixed \(n\),
\[
\binom{2N}{N} \sim \frac{4^N}{\sqrt{\pi N}}, \qquad 
\binom{2N-1}{N-n} \sim \frac{4^N}{2\sqrt{\pi N}}
\qquad (N \to \infty).
\]
Therefore,
\[
w_{1n}^{[N]} \sim 
\frac{(-1)^{n-1} \, 2N}{4^{2N-1}(2n-1)} \cdot 
\frac{4^N}{\sqrt{\pi N}} \cdot \frac{4^N}{2\sqrt{\pi N}}
= \frac{4(-1)^{n+1}}{\pi(2n-1)}.
\]

\begin{table}[ht]
\centering
\[
\begin{array}{c|c|c|}
u_n & w_{1n}^{[N]} & w_{1n}^{[\infty]} \\
\hline
n^2 & 
2(-1)^{n+1} \cdot \dfrac{\binom{2N}{N+n}}{\binom{2N}{N}} & 
2(-1)^{n+1} \\
(n-\tfrac12)^2 & 
\dfrac{(-1)^{n-1}\,2N}{4^{2N-1}(2n-1)}
   \binom{2N}{N}\binom{2N-1}{N-n} & 
\dfrac{4(-1)^{n+1}}{\pi(2n-1)}
\end{array}
\]
\caption{Explicit expressions for the first-row coefficients 
\(w_{1n}^{[N]}\) and their limits \(w_{1n}^{[\infty]}\) for the two 
node sequences.}
\label{tab:w1n-values}
\end{table}

\section{A Ratio of Determinants and Its Asymptotic Behavior}
\label{sec:ratio-determinant}

In this section we identify the limit of the ratio~\eqref{eq:ratio}
defined in the introduction for the two quadratic node sequences
\(u_n=n^2\) and \(u_n=(n-\tfrac12)^2\).

It is immediate from the definition that \(R_N(0;\bu)=1\),
and that \(R_N(s;\bu)=0\) for \(s=1,\ldots,(N-1)\). Moreover, if we denote
\[
\alpha^{[N]}_k(s) \eqdef \sum_{n=1}^N w_{kn}^{[N]} u_n^{-s},\qquad k=1,\ldots,N,
\]
we have the following evaluation:
\begin{equation*}
R_N(s;\bu) = \Det{W_{N}(\bu) V_N^{(s)}(\bu)}
= \det{
\begin{pmatrix}
 \alpha^{[N]}_1(s) & 0 & 0 &\ldots & 0 \\
 \alpha^{[N]}_2(s) & 1 & 0 &\ldots & 0 \\
\vdots             & \vdots & \vdots & \ddots  & \vdots \\
\alpha^{[N]}_N(s)  & 0 & 0 &\ldots & 1
\end{pmatrix}
}
 = \sum_{n=1}^N  \left(\prod_{\substack {1\le j\le N\\j\ne n}} \frac{u_j}{u_j-u_n}\right)
 \frac{1}{u_n^{s}}.
\end{equation*}
For the last equality we use the expression of \(w_{1n}^{[N]}\) given in
Appendix~\ref{app:inverseVandermonde}.

Recall the definitions of the Dirichlet eta and beta functions:
\[
\eta(s) \eqdef \sum_{n=1}^\infty \frac{(-1)^{n+1}}{n^s},
\qquad
\beta(s) \eqdef \sum_{m=0}^\infty \frac{(-1)^m}{(2m+1)^s}
= 2^{-s}\sum_{m=1}^\infty \frac{(-1)^{m+1}}{(m-\tfrac12)^s}.
\]
We can identify the limits of the series \(\alpha_k^{[N]}(s)\)
using the limiting weights \(w_{kn}^{[\infty]}\) computed in
Section~\ref{sec:InverseVandermonde} and made explicit in
Corollary~\ref{cor:Wkn-closed-forms}
(the particular case $k=1$ is in Table~\ref{tab:w1n-values}).
These series converge for \(\Re(s)>0\), and are collected in
Table~\ref{tab:S-infty}.
The values obtained by replacing \(w_{kn}^{[N]}\)
with \(w_{kn}^{[\infty]}\) in the definition of \(\alpha_k^{[N]}(s)\),
are the starting point of the shift representation theory developed
in Section~\ref{sec:shift-structure}.
\begin{table}[htb!]
\centering
\[
\begin{array}{c|c|c|}
u_n & \alpha_1^{[\infty]}(s) & \alpha_k^{[\infty]}(s) \\
\hline
n^2 & 2\,\eta(2s) & 
2 \displaystyle\sum_{l=1}^{k} (-1)^{l+1} \dfrac{\pi^{2l-2}}{(2l-1)!} \,
\eta(2s+2k-2l) \\
(n-\tfrac12)^2 &  \dfrac{2^{2s+2}}{\pi}\,\beta(2s+1) & 
\dfrac{2^{2s+1}}{\pi}\displaystyle\sum_{l=1}^{k} (-1)^{l+1} \dfrac{\pi^{2l-2}}{(2l-2)!} \, 2^{\,2k-2l+1} \, \beta(2s+2k-2l+1)
\end{array}
\]
\caption{Limiting functions \(\alpha_k^{[\infty]}(s)\) for
\(u_n=n^2\) and \(u_n = (n-\tfrac12)^2\).}
\label{tab:S-infty}
\end{table}

Moreover, for \(k=1\) the convergence of \(R_N(s;\bu)=\alpha_1^{[N]}(s)\)
to its limit can be shown to be uniform on compact subsets of \(\C\),
for both node sequences.
It follows from Lemmas~4 and~5 of \cite{biane2001probability}, applied as
in the proof of Theorem~3 of the same reference. We sketch the arguments.
For $\nu\in\{0,\tfrac12\}$, set $a_n^{(0)}=n^{-2}$,
$a_n^{(1/2)}=(n-\tfrac12)^{-2}$, and define, as in
\cite{biane2001probability},
\[
\Pi_{n,N}^{(\nu)}\eqdef\prod_{\substack{1\le j\le N\\ j\ne n}}
\Bigl(1-\frac{a_n^{(\nu)}}{a_j^{(\nu)}}\Bigr).
\]
Some algebra gives (see Table~\ref{tab:w1n-values}),
\[
\frac{1}{\Pi_{n,N}^{(0)}}
=2(-1)^{n+1}\frac{\binom{2N}{N-n}}{\binom{2N}{N}},
\qquad\text{and}\qquad
\frac{1}{\Pi_{n,N}^{(1/2)}}
=\frac{(-1)^{n-1}2N}{4^{2N-1}(2n-1)} \binom{2N}{N}\binom{2N-1}{N-n},
\]
so that in both cases $w_{1n}^{[N]}=1/\Pi_{n,N}^{(\nu)}$.

Let $X_N^{(\nu)}=\sum_{n=1}^N a_n^{(\nu)}\epsilon_n$ with independent
standard exponential $\epsilon_n\sim\mathrm{Exp}(1)$.
Lemma~5 of \cite{biane2001probability} gives
\[
\mathbb{E}\bigl[(X_N^{(\nu)})^s\bigr]
=\Gamma(s+1)\,\alpha_1^{[N]}(s),
\]
and Lemma~4 gives that $\mathbb{E}[(X_N^{(\nu)})^s]$ converges as
$N\to\infty$, uniformly on compact subsets of \(\C\), to an entire
function. Since \(1/\Gamma(s+1)\) is entire, and therefore bounded on
compact sets, the quotient
\[
\alpha_1^{[N]}(s)
=
\frac{\mathbb{E}[(X_N^{(\nu)})^s]}{\Gamma(s+1)}
\]
also converges uniformly on compact subsets of \(\C\). In the half-plane
$\Re(s)>0$, the limit is identified with
$\sum_{n=1}^{\infty}w_{1n}^{[\infty]}u_n^{-s}$, i.e. $2\eta(2s)$ for
$u_n=n^2$ and $\frac{2^{2s+2}}{\pi}\beta(2s+1)$ for $u_n=(n-\tfrac12)^2$. Both sides being entire in $s$, the identity
extends to all of $\C$ by analytic continuation.

\section{Shift Representation of the Dirichlet Eta and Beta Functions}
\label{sec:shift-structure}

Let \(s\in\C\). In this section we investigate the properties
of the series \(\alpha_k^{[\infty]}\), \(k\geq1\), for the two
choices \(u_n=n^2\) and \(u_n=(n-\tfrac12)^2\).
To lighten the notation we write
\[
\alpha_0(k,s)\quad\text{when }u_n=n^2,
\qquad
\alpha_{\frac12}(k,s)\quad\text{when }u_n=\left(n-\tfrac12\right)^2,
\]
and we drop the overall factor \(2\) which appears in both families.
We first derive a recurrence and a generating function
(Subsection~\ref{subsec:recurrence}), then invert the recurrence to
express the first row \(\alpha_\nu(1,s+n)\) in terms of the column
\(\alpha_\nu(k+1,s)\) (Subsection~\ref{subsec:inversion-theorem}).
Substituting the integral forms of the \(\alpha_\nu\)'s and exchanging 
um and integral gives the result of Corollary~\ref{cor:eta-beta-clausen}.
Finally we
state some miscellaneous related results in Subsection~\ref{subsec:miscellaneous}.

Using the integral representations \eqref{eq:Wkn-infty-0},
together with the definitions of the generalized Clausen functions
from Appendix~\ref{app:Clausen}, we obtain
\begin{multline}\label{eq:alphaks_0}
\alpha_0(k,s) = \frac{(-1)^{k+1} }{\pi}\sum_{n=1}^{\infty} \frac{n^{-2s+1}}{(2k-1)!}
\int_0^\pi \theta^{2k-1} \sin(n\theta)\,d\theta\\
= \frac{(-1)^{k+1}}{(2k-1)!}\,\frac{1}{\pi}
\int_{0}^{\pi}\theta^{2k-1}\,\Sl_{2s-1}(\theta)\,d\theta,
\quad k\in\N^\star,\quad \RE(s)\ge\tfrac12,
\end{multline}
the interchange of sum and integral being justified for \(\RE(s)>1\)
and extended to $\RE(s)\geq1/2$ as the integral is defined in this range.
Similarly, using the integral
representations \eqref{eq:Wkn-infty-12} together with the definitions of the
generalized half-Clausen functions from Appendix~\ref{app:Half-Clausen},
we have
\begin{multline}\label{eq:alphaks_12}
\alpha_{\frac12}(k,s)
=  \frac{(-1)^{k+1} }{\pi}\sum_{n=1}^{\infty} \frac{(n-\tfrac12)^{-2s+1}}{(2k-1)!}
\int_0^{\pi} \theta^{2k-1} \sin\left((n-\tfrac12)\theta\right) \, d\theta\\
= \frac{(-1)^{k+1}}{(2k-1)!}\,\frac{1}{\pi}
\int_{0}^{\pi}\theta^{2k-1}\,\Slh_{2s-1}(\theta)\,d\theta,
\quad k\in\N^\star,\quad \RE(s)\ge\tfrac12.
\end{multline}
These integral forms of the \(\alpha_\nu\)'s are the starting point for all the
evaluations in this section. Before going to the main
results of this section, we record some particular values, mainly
when \(k=1\) for general $s$, and for \(s=\tfrac12\) and \(s=1\) with
$k$ arbitrary.
\begin{itemize}
\item When \(k=1\) we have (see Table~\ref{tab:S-infty}),
\[
\alpha_0(1,s)=\eta(2s),\qquad
\alpha_{\frac12}(1,s)=\frac{2^{2s+1}}{\pi}\,\beta(2s+1).
\]
\item At \(s=\tfrac12\) the Clausen functions reduce to elementary
trigonometric functions and we get
\begin{equation}\label{eq:alpha-moments-correct}
\alpha_0\!\left(k,\tfrac12\right)
=\frac{(-1)^{k+1}}{(2k-1)!2\pi}\,\int_0^\pi \theta^{2k-1} \cot\frac{\theta}{2}\,d\theta,
\qquad
\alpha_{\frac12}\!\left(k,\tfrac12\right)
=\frac{(-1)^{k+1}}{(2k-1)!2\pi}\,\int_0^\pi \theta^{2k-1} \csc\frac{\theta}{2}\,d\theta.
\end{equation}
These integrals are classical; they can be evaluated in
terms of zeta and beta values, respectively, by expanding the
trigonometric kernels (see \cite[entries 3.747 and 3.748]{gradshteyn2014}).

\item At \(s=1\), \(\Sl_1(\theta)=\frac{\pi-\theta}{2}\) and
\(\Slh_1(\theta)=\frac{\pi}{2}\) for \(\theta\in(0,\pi)\), so the integrals
in \eqref{eq:alphaks_0}--\eqref{eq:alphaks_12} are elementary. One obtains
\[
\alpha_0(k,1)=\frac{(-1)^{k+1}\pi^{2k}}{2(2k+1)!},
\qquad
\alpha_{\frac12}(k,1)=\frac{(-1)^{k+1}\pi^{2k}}{2(2k)!}.
\]
\end{itemize}

\subsection{Recurrence and Generating Function}
\label{subsec:recurrence}
Integrating by parts twice in \eqref{eq:alphaks_0} and \eqref{eq:alphaks_12}
yields the unified recurrence
\begin{equation}\label{eq:alpha-rec-unified}
\alpha_\nu(k,s)
=\frac{(-1)^{k-1}\pi^{2k-2}}{(2k-1-2\nu)!}\,\alpha_\nu(1,s)
+\alpha_\nu(k-1,s+1),
\qquad k\ge2,\;\RE(s)\ge1,
\end{equation}
where $\nu\in\{0,\tfrac12\}$ and $(2k-1-2\nu)!$ is $(2k-1)!$
for $\nu=0$, and $(2k-2)!$ for $\nu=\tfrac12$.
For $\nu=0$ the boundary term arises from the first primitive
$-\Cl_{2s}(\theta)$ evaluated at $\theta=\pi$, while for
$\nu=\tfrac12$ it arises from the second primitive $\Slh_{2s+1}(\pi)$;
the remaining integral is proportional to $\alpha_{\nu}(k-1,s+1)$
in both cases. Unfolding \eqref{eq:alpha-rec-unified} gives
\begin{equation}\label{eq:alpha-closed-unified}
\alpha_\nu(k,s)=\sum_{j=0}^{k-1}
\frac{(-1)^{k-j-1}\,\pi^{2k-2j-2}}{\left(2k-2j-1-2\nu\right)!}\;
\alpha_\nu(1,s+j).
\end{equation}
This expresses every row $k$ as a finite linear combination of the first
row at shifted values of $s$, with explicit polynomial coefficients in $\pi^2$.
Looking closely at these coefficients, we observe that they involve
the Taylor coefficients of $\frac{\sin(\pi t)}{\pi t}$
when $\nu = 0$, and $\cos(\pi t)$ when $\nu=\tfrac12$.
This leads to the following proposition.
\begin{proposition}\label{prop:generating functions}
Let $s\in\C$ and define the ordinary generating functions
\[
G_{s}^{(\nu)}(t) \eqdef \sum_{k=0}^{\infty}\alpha_\nu(k+1,s)\,t^{2k},
\qquad \nu\in\{0,\tfrac12\}.
\]
Then
\begin{equation}\label{eq:generating functions}
G_{s}^{(0)}(t)=\frac{\sin(\pi t)}{\pi t}\,
\sum_{j=0}^{\infty}\alpha_{0}(1,s+j)\,t^{2j},
\qquad
G_{s}^{(1/2)}(t)=\cos(\pi t)\,
\sum_{j=0}^{\infty}\alpha_{\frac12}(1,s+j)\,t^{2j}.
    \end{equation}
\end{proposition}
\begin{proof}
We first prove the identity for \(\nu=0\). Substituting the unfolded
formula \eqref{eq:alpha-closed-unified} in the definition, we have for $\nu=0$,
\[
G_{s}^{(0)}(t)
= \sum_{k=0}^{\infty} \sum_{j=0}^{k} \frac{(-1)^{k-j}\,\pi^{2k-2j}}{(2k-2j+1)!}\,
\alpha_0(1,s+j) t^{2k}.
\]
The double sum runs over $0\leq j\leq k$, which is equivalent to
$j\geq0$ and $m=k-j\ge0$. Substituting $k=m+j$, we obtain
\[
G_{s}^{(0)}(t)
=\sum_{j=0}^{\infty}\alpha_0(1,s+j)\,t^{2j}
\sum_{m=0}^{\infty}
\frac{(-1)^m(\pi t)^{2m}}{(2m+1)!}.
\]
The inner sum is precisely the Taylor expansion of \(\sin(\pi t)/(\pi t)\);
hence the result follows when $\nu=0$.
For \(\nu=\tfrac12\), the same computations lead to
\[
G_{s}^{(1/2)}(t)
=\sum_{j=0}^{\infty}\alpha_{\frac12}(1,s+j)\,t^{2j}
\sum_{m=0}^{\infty}
\frac{(-1)^m(\pi t)^{2m}}{(2m)!}.
\]
The inner sum is the Taylor expansion of \(\cos(\pi t)\)
which completes the proof.
\end{proof}

\subsection{Inversion theorem}
\label{subsec:inversion-theorem}

In the previous part, we expressed the family of functions $\alpha_\nu(k,s)$
for $k\in\N^\star$ as a linear combination of the family of functions
$\alpha_\nu(1,s+j)$ for $j\ge0$. In this part, we invert this relation
and express the functions $\alpha_\nu(1,s+j)$ as integrals of the Clausen
functions against certain polynomials.
\begin{theorem}[Inversion theorem]\label{th:inv-general-final}
For any $s\in\C$ with $\RE(s)\ge1/2$, we have
\begin{equation}\label{eq:inv-general-0}
\alpha_0(1,s+n) =\sum_{k=0}^{n} (-1)^{n-k+1}
\frac{2(2^{2n-2k-1}-1)B_{2n-2k}}{(2n-2k)!}\,\pi^{2n-2k}\,\alpha_0(k+1,s),
\quad n\geq0,
\end{equation}
and
\begin{equation}\label{eq:inv-general-12}
\alpha_{\tfrac12}(1,s+n) =\sum_{k=0}^{n} (-1)^{n-k} \frac{E_{2n-2k}}{(2n-2k)!}\,\pi^{2n-2k}\,\alpha_{\tfrac12}(k+1,s), \quad n\geq 0.
\end{equation}
\end{theorem}
\begin{proof}
We treat the case $\nu=0$ first. In Proposition~\ref{prop:generating functions},
multiplying both sides of equation \eqref{eq:generating functions} by
$\pi t/\sin(\pi t)$ gives
\[
\frac{\pi t}{\sin(\pi t)}\sum_{k=0}^{\infty}\alpha_0(k+1,s)\,t^{2k}
=\sum_{j=0}^{\infty}\alpha_0(1,s+j)\,t^{2j}.
\]
Using the Taylor series of the cosecant function (\cite[Entry 22.26]{spiegel2009mathematical})
we obtain
\[
\left(\sum_{n=0}^{\infty}(-1)^{n+1}\,\frac{2(2^{2n-1}-1)B_{2n}}{(2n)!}\,(\pi t)^{2n} \right)
\left(\sum_{k=0}^\infty\alpha_0(k+1,s) t^{2k}\right)
= \sum_{n=0}^{\infty}\alpha_0(1,s+n)\,t^{2n},\quad |t|<1.
\]
Identifying the coefficient of $t^{2n}$ in the product gives \eqref{eq:inv-general-0}
and completes the proof.
Similarly, for $\nu=\tfrac12$, in Proposition~\ref{prop:generating functions}
dividing by $\cos(\pi t)$ on both sides gives
\[
\frac{1}{\cos(\pi t)}\sum_{k=0}^{\infty}\alpha_{\frac12}(k+1,s)\,t^{2k}
=\sum_{j=0}^{\infty}\alpha_{\frac12}(1,s+j)\,t^{2j}.
\]
Using the Taylor expansion of the secant function (\cite[Entry 22.25]{spiegel2009mathematical}),
we obtain
\[
\left(\sum_{n=0}^{\infty}(-1)^{n}\,
\frac{E_{2n}}{(2n)!}\,(\pi t)^{2n}\right)
\left(\sum_{k=0}^{\infty}\alpha_{\frac12}(k+1,s)\,t^{2k}\right)
=\sum_{n=0}^{\infty}\alpha_{\frac12}(1,s+n)\,t^{2n},\quad |t|<\tfrac12.
\]
Identifying the coefficient of \(t^{2n}\) in this product gives
\eqref{eq:inv-general-12}.
\end{proof}
From this inversion theorem, we will derive the main result of this section.
Define the polynomials
\begin{eqnarray}
P^{(0)}_{2n+1}(\theta) &=& \sum_{k=0}^{n} \binom{2n+1}{2k+1}
(1-2^{2n-2k-1})B_{2n-2k}\,\theta^{2k+1},\label{eq:Pn0-def-final} \\ 
P^{(1/2)}_{2n+1}(\theta) &=& \sum_{k=0}^{n} \binom{2n+1}{2k+1}
E_{2n-2k}\,\theta^{2k+1}.\label{eq:Pn12-def-final}
\end{eqnarray}
Interchanging the finite sum and the integral and making the change of
variable \(t=\theta/\pi\) gives
\begin{align*}
\alpha_0(1,s+n) &= \frac{2(-1)^{n}\pi^{2n}}{(2n+1)!}
\int_{0}^{\pi} P_{2n+1}^{(0)}\!\left(\frac{\theta}{\pi}\right) \Sl_{2s-1}(\theta)\,d\theta
= \frac{2(-1)^{n}\pi^{2n+1}}{(2n+1)!}
\int_{0}^{1} P_{2n+1}^{(0)}(t)\, \Sl_{2s-1}(\pi t)\,dt,
\\
\alpha_{\frac12}(1,s+n) &= \frac{(-1)^{n}\pi^{2n}}{(2n+1)!}
\int_{0}^{\pi} P_{2n+1}^{(1/2)}\!\left(\frac{\theta}{\pi}\right) \Slh_{2s-1}(\theta)\,d\theta
= \frac{(-1)^{n}\pi^{2n+1}}{(2n+1)!}
\int_{0}^{1} P_{2n+1}^{(1/2)}(t)\, \Slh_{2s-1}(\pi t)\,dt.
\end{align*}
The first few of these polynomials are given in
Table~\ref{tab:polynomials-P} together with the even polynomials.

\begin{remark*}
The polynomials \(P_{2n+1}^{(\nu)}\), defined in
equations~\eqref{eq:Pn0-def-final} and~\eqref{eq:Pn12-def-final}, are
rescaled Bernoulli and Euler polynomials. More precisely, for every
\(n\ge0\),
\begin{equation}\label{eq:P-BE}
P_{2n+1}^{(0)}(t)=2^{2n}B_{2n+1}\!\left(\frac{t+1}{2}\right),
\qquad
P_{2n+1}^{(1/2)}(t)=2^{2n+1}E_{2n+1}\!\left(\frac{t+1}{2}\right).
\end{equation}
For the Bernoulli family, the identity follows from the multiplication
theorem for Bernoulli polynomials in the form
\[
B_{2n+1}(t)=2^{2n}\left[B_{2n+1}\!\left(\frac t2\right)
+B_{2n+1}\!\left(\frac{t+1}{2}\right)\right],
\]
which gives directly
\[
B_{2n+1}(t)-2^{2n}B_{2n+1}\!\left(\frac t2\right)
=
2^{2n}B_{2n+1}\!\left(\frac{t+1}{2}\right) = P_{2n+1}^{(0)}(t).
\]
For the Euler family, the Taylor expansion of \(E_{2n+1}(x)\) about
\(x=\tfrac12\) reads\[
E_{2n+1}(x) = \sum_{j=0}^{2n+1}\binom{2n+1}{j}
E_{2n+1-j}\!\left(\frac12\right)
\left(x-\frac12\right)^{j}
= \sum_{j=0}^{2n+1}\binom{2n+1}{j}
2^{-(2n+1-j)}E_{2n+1-j}
\left(x-\frac12\right)^{j},
\]
where we used \(E_m(1/2)=2^{-m}E_m\). Since \(E_m=0\) for odd \(m\), the terms with even \(j\) vanish. Setting \(x=(t+1)/2\), so that \(x-1/2=t/2\), and writing \(j=2k+1\), we obtain
\[
E_{2n+1}\!\left(\frac{t+1}{2}\right)
= 2^{-(2n+1)}
\sum_{k=0}^{n}\binom{2n+1}{2k+1}E_{2n-2k}\,t^{2k+1}.
\]
Multiplying by \(2^{2n+1}\) and comparing with the definition of \(P_{2n+1}^{(1/2)}\) gives the identity.

\end{remark*}

As a direct consequence of Theorem~\ref{th:inv-general-final} and the
identification \eqref{eq:P-BE}, we obtain the following integral
representations for the shifted eta and beta functions.
\begin{corollary}[Integral form of the shift representation]
\label{cor:eta-beta-clausen}
For every \(s\in\C\) with \(\RE(s)\ge\tfrac12\) and every integer \(n\ge0\),
\begin{equation}\label{eq:eta-clausen}
\eta(2s+2n)=\frac{2^{2n+1}(-1)^{n}\pi^{2n+1}}{(2n+1)!}
\int_{0}^{1} B_{2n+1}\!\left(\frac{t+1}{2}\right) \Sl_{2s-1}(\pi t)\,dt,
\end{equation}
and
\begin{equation}\label{eq:beta-clausen}
\beta(2s+2n+1)=\frac{(-1)^{n}\pi^{2n+2}}{(2n+1)!}
\int_{0}^{1} E_{2n+1}\!\left(\frac{t+1}{2}\right) \frac{\Slh_{2s-1}(\pi t)}{2^{2s}}\,dt.
\end{equation}
\end{corollary}

\begin{remark*}
The corollary is stated for \(\RE(s)\ge\tfrac12\), but the shift
representation actually holds for the wider range \(\RE(s)>0\).
The special case $s=\tfrac{1}{2}$ is discussed below.
The special case $\RE(s)=\tfrac{1}{2}$  with $\Im(s)\neq0$
and the extension to the strip $0<\RE(s)<\tfrac12$ are discussed in
Section~\ref{sec:unified-kernel} in a wider setup.
\end{remark*}

\paragraph{Consistency check} We verify that the general shift representations of
Corollary~\ref{cor:eta-beta-clausen} reduce to the classical formulas
for the eta values at odd integers and the beta values at even integers.
\begin{proposition}[Classical odd/even integer case]
\label{prop:odd-integer-equivalence}
For every integer \(k\ge1\), the integral form of
Corollary~\ref{cor:eta-beta-clausen} at \(s=\tfrac12\) reduce to the
classical formulas
\[
\eta(2k+1)=(-1)^{k}\frac{(1-2^{2k})\pi^{2k+1}}{(2k+1)!}
\int_0^1 B_{2k+1}(t)\cot(\pi t)\,dt,
\]
and
\[
\beta(2k)=(-1)^k\frac{\pi^{2k}}{4(2k-1)!}
\int_0^1 E_{2k-1}(t)\sec(\pi t)\,dt.
\]
\end{proposition}

\begin{proof}
At \(s=\tfrac12\), the kernels are
\(\Sl_0(\theta)=\frac12\cot(\theta/2)\) and
\(\Slh_0(\theta)=\frac12\csc(\theta/2)\)
(Appendices~\ref{app:Clausen} and~\ref{app:Half-Clausen}).

For the eta case, \eqref{eq:eta-clausen} with \(s=\tfrac12\) and \(n=k\) gives
\[
\eta(2k+1)=\frac{2^{2k+1}(-1)^{k}\pi^{2k+1}}{(2k+1)!}
\int_0^1 B_{2k+1}\!\left(\frac{t+1}{2}\right)\frac{1}{2}\cot\frac{\pi t}{2}\,dt .
\]
Substituting \(u=(t+1)/2\) and using
\(\cot(\pi(u-1/2))=-\tan(\pi u)\), together with the identity
\(\tan(\pi u)=\cot(\pi u)-2\cot(2\pi u)\) and the duplication formula
\(B_{2k+1}(u)=2^{-2k}\left[B_{2k+1}(u/2)+B_{2k+1}((u+1)/2)\right]\), the
integral reduces to
\[
\int_0^1 B_{2k+1}\!\left(\frac{t+1}{2}\right)\cot\frac{\pi t}{2}\,dt
=
\frac{1-2^{2k}}{2^{2k}}\int_0^1 B_{2k+1}(t)\cot(\pi t)\,dt .
\]
Substituting this back into \eqref{eq:eta-clausen} yields the stated formula
for \(\eta(2k+1)\).

For the beta case, the substitution \(u=(t+1)/2\) and
\(\csc(x-\pi/2)=-\sec x\) give
\[
\int_0^1 E_{2k-1}\!\left(\frac{t+1}{2}\right)\Slh_0(\pi t)\,dt
=-\int_{1/2}^1 E_{2k-1}(u)\sec(\pi u)\,du
=-\frac12\int_0^1 E_{2k-1}(u)\sec(\pi u)\,du,
\]
since \(E_m(1-t)=(-1)^mE_m(t)\) and \(\sec(\pi(1-t))=-\sec(\pi t)\), so the
integrand is symmetric about \(u=\tfrac12\). Substituting this into
\eqref{eq:beta-clausen} with \(s=\tfrac12\) and \(n=k-1\) yields the stated
formula for \(\beta(2k)\).
\end{proof}

\subsection{Complementary Results}
\label{subsec:miscellaneous}
In this subsection, we first show that an elementary integration by part transform
the families of polynomial and Kernels $(P_{2n+1}^{(\nu)},\Sl_{2s-1}^{(\nu)})$ gives
the alternative $(P_{2n}^{(\nu)},\Cl_{2s}^{(\nu)})$ for $\nu\in\set{0, 1/2}$.
In the second part we list without proof the main properties of the family of
polynomials $(P_n^{(\nu)})_{n\geq0}$ for $\nu\in\set{0,\tfrac12}$.

\paragraph{Extension to even orders}

For \(\nu\in\{0,\tfrac12\}\) and \(n\ge1\), set
\begin{equation}\label{eq:even-extension}
P_{2n}^{(\nu)}(t) \eqdef \frac{1}{2n+1}P_{2n+1}^{(\nu)\prime}(t).
\end{equation}
By construction, \(P_{2n}^{(\nu)}\) is an even polynomial of degree \(2n\).

The moments of the even polynomials against the cosine Clausen function are
obtained from the odd moments by integration by parts. Since
\(\frac{d}{d\theta}\Cl_{2s}^{(\nu)}(\theta)=-\Sl_{2s-1}^{(\nu)}(\theta)\), we have
\[
\int_0^1 P_{2n+1}^{(\nu)}(t)\,\Sl_{2s-1}^{(\nu)}(\pi t)\,dt
=
-\frac{1}{\pi}\left[P_{2n+1}^{(\nu)}(t)\,\Cl_{2s}^{(\nu)}(\pi t)\right]_0^1
+
\frac{2n+1}{\pi}
\int_0^1 P_{2n}^{(\nu)}(t)\,\Cl_{2s}^{(\nu)}(\pi t)\,dt.
\]
The boundary term vanishes in both families. At \(t=0\),
\(P_{2n+1}^{(\nu)}(0)=0\) for every \(n\ge0\) and both \(\nu\). At \(t=1\),
\(P_{2n+1}^{(0)}(1)=0\) for \(n\ge1\), while the
half-shifted cosine Clausen function satisfies \(\Clh_{2s}(\pi)=0\).
Hence the moment identity
\[
\int_0^1 P_{2n+1}^{(\nu)}(t)\,\Sl_{2s-1}^{(\nu)}(\pi t)\,dt
=
\frac{2n+1}{\pi}
\int_0^1 P_{2n}^{(\nu)}(t)\,\Cl_{2s}^{(\nu)}(\pi t)\,dt
\]
holds for \(n\ge1\) and \(\nu\in\{0,\tfrac12\}\).
Substituting into the odd moments of
Corollary~\ref{cor:eta-beta-clausen} gives the even moments
\begin{align}
\eta(2s+2n) &=\frac{2(-1)^n\pi^{2n}}{(2n)!}
\int_0^1 P_{2n}^{(0)}(t)\,\Cl_{2s}(\pi t)\,dt,
\label{eq:even-moment-eta}\\[4pt]
\beta(2s+2n+1) &=\frac{(-1)^n\pi^{2n+1}}{2^{2n}(2n)!}
\int_0^1 P_{2n}^{(1/2)}(t)\, \frac{\Clh_{2s}(\pi t)}{2^{2s+1}}\,dt.
\label{eq:even-moment-beta}
\end{align}
The first polynomials are given in the Table~\ref{tab:polynomials-P} below.
\begin{table}[htb!]
\centering
\small
\setlength{\tabcolsep}{4pt}
\[
\begin{array}{c|c|c|c|c}
n & P_{2n}^{(0)}(t) & P_{2n+1}^{(0)}(t) & P_{2n}^{(1/2)}(t) & P_{2n+1}^{(1/2)}(t) \\
\hline
0 & \tfrac12 & \tfrac12 t & 1 & t \\[6pt]
1 & -\tfrac16 + \tfrac12 t^2 & -\tfrac12 t + \tfrac12 t^3 & -1 + t^2 & -3t + t^3 \\[6pt]
2 & \tfrac{7}{30} - t^2 + \tfrac12 t^4 & \tfrac76 t - \tfrac53 t^3 + \tfrac12 t^5 
& 5 - 6t^2 + t^4 & 25t - 10t^3 + t^5 \\[6pt]
3 & -\tfrac{31}{42} + \tfrac72 t^2 - \tfrac52 t^4 + \tfrac12 t^6 & -\tfrac{31}{6}t + \tfrac{49}{6}t^3 - \tfrac72 t^5 + \tfrac12 t^7 & -61 + 75t^2 - 15t^4 + t^6 & -427t + 175t^3 - 21t^5 + t^7
\end{array}
\]
\caption{The polynomials \(P_n^{(0)}\) and \(P_n^{(1/2)}\) for \(n=0,1,2\).
The odd order are defined by \eqref{eq:Pn0-def-final} and \eqref{eq:Pn12-def-final}.
The even orders are defined by \eqref{eq:even-extension}.
}
\label{tab:polynomials-P}
\end{table}
\paragraph{Properties of the Polynomials}
\label{subsec:properties}

The polynomials \(P_n^{(0)}\) and \(P_n^{(1/2)}\), \(n\ge1\), enjoy the
following elementary properties.

\begin{enumerate}
\item \emph{Relation to Bernoulli and Euler polynomials}
(equation~\eqref{eq:P-BE})
$$
P_n^{(0)}(t) = \frac{1}{2} 2^n B_n\left(\frac{t+1}{2} \right),
\qquad
P_n^{(1/2)}(t) = 2^n E_n\left(\frac{t+1}{2} \right).
$$
\item \emph{Values at \(t=0\).}
\begin{equation*}
\begin{array}{lclr}
P_{2n+1}^{(0)}(0)=0,& \quad & P_{2n+1}^{(1/2)}(0)=0,& \quad n\ge0,\\
P_{2n}^{(0)}(0)=(1-2^{2n-1})B_{2n}, &\quad &P_{2n}^{(1/2)}(0)=E_{2n},& \quad n\ge0.
\end{array}
\end{equation*}

\item \emph{Values at \(t=1\).}
\begin{equation*}
\begin{array}{lrclr}
P_{2n+1}^{(0)}(1)=0,& n\geq1, & & 
P_{2n+1}^{(1/2)}(1)=\dfrac{2^{2n+2}\left(2^{2n+2}-1\right)}{2n+2}\,B_{2n+2}, & n\ge0,\\
P_{2n}^{(0)}(1)=2^{2n-1}B_{2n},& n\geq0, & &
  P_{2n}^{(1/2)}(1)=0, & n\ge1,
\end{array}
\end{equation*}
and the exceptional values
\[
P_1^{(0)}(1)=\frac12,\qquad P_0^{(1/2)}(1)=1.
\]

\item \emph{Derivatives at \(t=0\)}. The even polynomials satisfy
\(P_{2n}^{(0)\prime}(0)=P_{2n}^{(1/2)\prime}(0)=0\), while for the
odd polynomials,
\[
P_{2n+1}^{(0)\prime}(0)=(2n+1)\left(1-2^{2n-1}\right)B_{2n}, \qquad
P_{2n+1}^{(1/2)\prime}(0)=(2n+1)E_{2n}.
\]

\item \emph{Fourier expansions.}
\begin{align}
P_{2n+1}^{(0)}(t) &=(-1)^{n}\frac{(2n+1)!}{\pi^{2n+1}}
\sum_{k=1}^{\infty}(-1)^{k+1}\frac{\sin(k\pi t)}{k^{2n+1}},
\qquad 0\le t<1,
\label{eq:Fourier-P0}\\
P_{2n}^{(0)}(t) &= (-1)^{n}\frac{(2n)!}{\pi^{2n}}
\sum_{k=1}^{\infty}(-1)^{k+1}\frac{\cos(k\pi t)}{k^{2n}},
\qquad 0\le t\le 1,
\label{eq:Fourier-P0-even}\\
P_{2n+1}^{(1/2)}(t) &= (-1)^{n}\frac{2(2n+1)!}{\pi^{2n+2}}
\sum_{k=1}^{\infty}(-1)^{k+1}
\frac{\sin\!\left((k-\tfrac12)\pi t\right)}{(k-\tfrac12)^{2n+2}},
\qquad 0\le t\le 1,
\label{eq:Fourier-P12}\\
P_{2n}^{(1/2)}(t) &= (-1)^{n}\frac{2(2n)!}{\pi^{2n+1}}
\sum_{k=1}^{\infty}(-1)^{k+1}
\frac{\cos\!\left((k-\tfrac12)\pi t\right)}{(k-\tfrac12)^{2n+1}},
\qquad 0\le t\le 1.
\label{eq:Fourier-P12-even}
\end{align}
For \eqref{eq:Fourier-P0}, the series converges to \(0\) at \(t=1\) while
\(P_1^{(0)}(1)=\tfrac12\). For \(n\ge1\), \(P_{2n+1}^{(0)}(1)=0\)
and the expansion holds on the closed interval.

\item \emph{Generating functions.}
\begin{align}
\sum_{k=0}^{\infty}P_k^{(0)}(t)\frac{u^k}{k!}
&= \frac{u\,e^{u(t+1)}}{e^{2u}-1} = \frac{u\,e^{ut}}{2\sinh(u)} ,
\qquad |\Im u|<\pi,
\label{eq:EGF-P0}\\
\sum_{k=0}^{\infty}P_k^{(1/2)}(t)\frac{u^k}{k!}
&= \frac{2\,e^{u(t+1)}}{e^{2u}+1}
= \frac{e^{u t}}{\cosh(u)}, \qquad |\Im u|<\frac{\pi}{2}.
\label{eq:EGF-P12}
\end{align}
These are, up to the factor \(1/2\) in the first identity, the classical
generating functions of the Bernoulli and Euler polynomials given in
\cite[Entries 24.2.3 and 24.2.8]{DLMF}) by replacing $x$ by $(t+2)/2$ and
$u$ by $2t$ in the cited reference.
\end{enumerate}

\section{Unified Shift Representation}
\label{sec:unified-kernel}

The corollary of Section~\ref{sec:shift-structure} expresses the shifted values
of the Dirichlet eta and beta functions as integrals of the polynomials
\(P_k^{(\nu)}\) against the (half-)Clausen kernels \(\Sl_{z-1}^{(\nu)}\).
This section extends the construction in two directions: the replacement of
the Clausen kernels by their polylogarithmic counterparts, the analytic
continuation that widens the range of the kernel parameter.

On the symmetric interval \([-1,1]\), the polynomials $P_{k}^{(\nu)}$ have
the parity of $k$: they are odd for $k$ odd and even for $k$ even. Combining
this observation with the parity of the Clausen functions allows us to unify
the previous results. Define the polylogarithm kernels
\begin{eqnarray*}
\Liz_{z-1}(\pi t)
&\eqdef& \Li_{z-1}\!\left(e^{i\pi t}\right)
=\Cl_{z-1}(\pi t)+i\,\Sl_{z-1}(\pi t),\\
\Lih_{z-1}(\pi t)
&\eqdef & e^{i\pi t/2}\,\Phi\!\left(e^{i\pi t},z-1,\tfrac12\right)
=\Clh_{z-1}(\pi t)+i\,\Slh_{z-1}(\pi t),
\end{eqnarray*}
with $\Li$ denoting the polylogarithm function and \(\Phi\) denoting the
Lerch transcendent. The cosine part is even and the
sine part is odd, so the parity of the polynomial selects exactly one of the
two components.

The kernels are initially defined by their convergent series for
\(\RE(z)\ge1\). Their analytic continuation to the half-plane
\(\RE(z)<1\) follows from the Jonquière representation of the Lerch
transcendent \cite[Entry 9.553]{gradshteyn2014}.
We did not find a relevant reference giving this integral representation specifically,
for completeness, the proof of the following continuation theorem is included
in appendix~\ref{app:jonquiere}.

\begin{theorem}[Analytic continuation]\label{thm:continuation}
For \(0<t<1\) and every \(z\in\C\) with \(\Re(z)<1\),
\begin{equation}\label{eq:Li-continuation}
\Liz_{z-1}(\pi t) = -\frac{1}{\pi^{2-z}}
\int_0^\infty \frac{u^{1-z}}{\sinh u}\,
\cosh\left((1-t)u-\frac{i\pi z}{2}\right)\,du,
\end{equation}
or equivalently,
\begin{eqnarray}\label{eq:Cl-Sl-continuation}
\Cl_{z-1}(\pi t)&=&-\frac{\cos\left(\frac{\pi z}{2}\right)}{\pi^{2-z}}
\int_0^\infty \frac{u^{1-z}}{\sinh u}\cosh((1-t)u)\,du,
\label{eq:Clz-continuation}\\
\Sl_{z-1}(\pi t)&=&\frac{\sin\left(\frac{\pi z}{2}\right)}{\pi^{2-z}}
\int_0^\infty \frac{u^{1-z}}{\sinh u}\sinh((1-t)u)\,du.
\label{eq:Slz-continuation}
\end{eqnarray}
For \(z\in\C\) with \(\Re(z)<2\),
\begin{equation}\label{eq:Lih-continuation}
\Lih_{z-1}(\pi t) = \frac{1}{\pi^{2-z}}
\int_0^\infty \frac{u^{1-z}}{\cosh u}\,
\sinh\left(\frac{i\pi z}{2}-(1-t)u\right)\,du,
\end{equation}
or equivalently,
\begin{eqnarray}
\Clh_{z-1}(\pi t) & =&
-\frac{\cos\left(\frac{\pi z}{2}\right)}{\pi^{2-z}}
\int_0^\infty \frac{u^{1-z}}{\cosh u} \sinh((1-t)u)\,du
\label{eq:Clh-continuation}\\
\Slh_{z-1}(\pi t) & =&
\frac{\sin\left(\frac{\pi z}{2}\right)}{\pi^{2-z}}
\int_0^\infty \frac{u^{1-z}}{\cosh u} \cosh((1-t)u)
du.\label{eq:Slh-continuation}
\end{eqnarray}
All the above integrals converge absolutely.
\end{theorem}

We now state the main theorem of this section.
\begin{theorem}\label{th:unified-kernels}
Let \(z\in\C\) with \(\RE(z)>0\).
For every odd integer \(k\ge1\), identity~\eqref{eq:unified-eta} holds
for every \(z\) with \(\RE(z)>0\):
\begin{equation}\label{eq:unified-eta}
e^{-ik\pi/2}\frac{\pi^k}{k!}
\int_{-1}^1 P_k^{(0)}(t)\,\Liz_{z-1}(\pi t)\,dt
= \eta(z+k-1).
\end{equation}
For every even integer \(k\ge2\), the same identity holds provided
that \(\RE(z)\ge1\).

For every odd integer \(k\ge1\), identity~\eqref{eq:unified-beta} holds
for every \(z\) with \(\RE(z)>0\):
\begin{equation}\label{eq:unified-beta}
e^{-ik\pi/2}\frac{\pi^{k+1}}{2^{k}k!}
\int_{-1}^1 P_k^{(1/2)}(t)\,
\frac{\Lih_{z-1}(\pi t)}{2^{z+1}}\,dt
= \beta(z+k).
\end{equation}
For every even integer \(k\ge0\), the same identity holds provided
that either \(\RE(z)>1\), or \(\RE(z)=1\) and \(z\ne1\).

The rotation \(e^{-ik\pi/2}\) selects the imaginary part of the
integral for odd \(k\) and the real part for even \(k\).
\end{theorem}

\begin{proof}
We distinguish the following cases.
\begin{itemize}
\item The case \(\RE(z)>1\) follows from
Corollary~\ref{cor:eta-beta-clausen} and
identities~\eqref{eq:even-moment-eta} and~\eqref{eq:even-moment-beta}.

\item If \(z=1\), the kernels are elementary:
\[
\Liz_0(\pi t)=-\frac12+\frac{i}{2}\cot\!\left(\frac{\pi t}{2}\right),
\qquad
\Lih_0(\pi t)=\frac{i}{2}\csc\!\left(\frac{\pi t}{2}\right).
\]
The odd case is treated in
Proposition~\ref{prop:odd-integer-equivalence}.  The even case is
immediate when \(\nu=0\), since the real part of the kernel is
constant and the polynomials \(P_k^{(0)}\) are integrable on
\([-1,1]\).  It is excluded when \(\nu=\tfrac12\), since the real part
of the kernel is zero while \(\beta(1+k)\ne0\).

\item The case \(z=1+ic\), \(c\ne0\), gives the Abel-summable series
\[
\Liz_{ic}(\pi t)=\sum_{m=1}^\infty \frac{e^{im\pi t}}{m^{ic}},
\qquad
\Lih_{ic}(\pi t)=\sum_{m=1}^\infty
\frac{e^{i(m-\tfrac{1}{2})\pi t}}{(m-\tfrac12)^{ic}}.
\]
Using the Fourier expansions of the polynomials given in
\eqref{eq:Fourier-P0}--\eqref{eq:Fourier-P12-even} and integrating
term by term, we recover \eqref{eq:unified-eta}
and~\eqref{eq:unified-beta}.

\item Finally, suppose \(0<\RE(z)<1\).  Even values of \(k\) are excluded,
so let \(k\) be odd.  Since the cosine parts are even and the sine parts
are odd, parity gives
\[
\int_{-1}^1 P_k^{(\nu)}(t)\,\Linu_{z-1}(\pi t)\,dt
=
2i\int_0^1 P_k^{(\nu)}(t)\,\Slnu_{z-1}(\pi t)\,dt .
\]
Write \(\sigma=\RE(z)\).  Since \(k\) is odd, \(P_k^{(\nu)}(t)=t\,Q(t)\)
with \(Q\) even and bounded on \([0,1]\), so \(|P_k^{(\nu)}(t)|\le Mt\).

For the integer case, \eqref{eq:Slz-continuation} gives
\[
\bigl|\Sl_{z-1}(\pi t)\bigr|
\le \frac{C}{\pi^{2-\sigma}}
\int_0^\infty u^{1-\sigma}\,
\frac{|\sinh((1-t)u)|}{\sinh u}\,du .
\]
Using \(|\sinh a|\le e^{|a|}\), \(\sinh u\ge u/2\) for \(u\in(0,1)\)
and \(\sinh u\ge e^u/4\) for \(u\ge1\), the ratio is
\(O\bigl(\min(1,e^{-tu})\bigr)\), so
\[
\int_0^1 |P_k^{(\nu)}(t)|\,
\frac{|\sinh((1-t)u)|}{\sinh u}\,dt
\le C'\!\int_0^1 t\,e^{-tu}\,dt
\le C'\min\!\left(\tfrac12,\frac{1}{u^2}\right).
\]
Therefore the double integral is dominated by
\[
C''\!\int_0^\infty u^{1-\sigma}
\min\!\left(1,u^{-2}\right)du<\infty,
\]
which converges for \(0<\sigma<1\).  The half-integer case is identical,
with \(\cosh u\) in place of \(\sinh u\) and \(\cosh((1-t)u)\) in place
of \(\sinh((1-t)u)\); the bounds are the same.

Fubini's theorem therefore applies directly on \((0,1)\times(0,\infty)\).
Evaluating the \(u\)-integral as in the case \(\RE(z)>1\) gives
\eqref{eq:unified-eta} and~\eqref{eq:unified-beta} for odd \(k\).\end{itemize}
\end{proof}

As a consequence, the shift representation in integral form has an
\emph{exceptional case} for \(\nu=\tfrac12\): when \(z=1\), it does not apply
for even \(k\), i.e. for the even polynomials \(P_k^{(1/2)}\).

\paragraph{Shift representation at odd polynomial orders}
Set \(L^{(0)}\eqdef\eta\), \(L^{(1/2)}\eqdef\beta\),
so that \(L^{(\nu)}\) denotes either function. The shift representation of
Theorem~\ref{th:unified-kernels}, restricted to odd polynomial orders
\(k=2n+1\), \(n\ge0\), computes the corresponding shifts to the
right of a base point \(z\in\C\) with \(0<\RE(z)<1\). Combined
with the functional equation, it generates every corresponding
integer shift of either function from two pairings and the factor
\(\Chi^{(\nu)}\).

Recall the functional equations
\[
L^{(\nu)}(w)=\Chi^{(\nu)}(w)\,L^{(\nu)}(1-w),
\]
with
\[
\Chiz(w) = 2^w \pi^{w-1}\sin\!\left(\frac{\pi w}{2}\right)\Gamma(1-w)
\frac{1-2^{1-w}}{1-2^w},
\qquad
\Chih(w) = \left(\frac{2}{\pi}\right)^{1-w}\Gamma(1-w)
\cos\!\left(\frac{\pi w}{2}\right),
\]
satisfying
\(\Chi^{(\nu)}(w)\Chi^{(\nu)}(1-w)=1\).
Write
\[
\cKz_u(\pi t)\eqdef\Liz_u(\pi t),
\qquad
\cKh_u(\pi t)\eqdef \frac{\Lih_u(\pi t)}{2^{\,u+2}},
\]
and
\[
\theta_{2n+1}^{(0)} \eqdef e^{-i(2n+1)\pi/2}\frac{\pi^{2n+1}}{(2n+1)!},
\qquad
\theta_{2n+1}^{(1/2)} \eqdef e^{-i(2n+1)\pi/2} \frac{\pi^{2n+2}}{2^{2n+1}(2n+1)!}.
\]
For every \(n\ge0\), Theorem~\ref{th:unified-kernels} can be written
compactly as
\[
L^{(\nu)}\bigl(z+2n+2\nu\bigr)
= \theta_{2n+1}^{(\nu)} \proscal{P_{2n+1}^{(\nu)}}{\cKnu_{z-1}}.
\]

\begin{itemize}
\item Fix \(0<\RE(z)<1\) with \(\RE(z)\ne\tfrac12\). At the base
point \(z\), the theorem gives
\[
L^{(\nu)}(z+2n+2\nu) = \theta_{2n+1}^{(\nu)}
\proscal{P_{2n+1}^{(\nu)}}{\cKnu_{z-1}},
\qquad n\ge0.
\]
At the reflected base point \(1-z\), the same formula holds with
\(\cKnu_{z-1}\) replaced by \(\cKnu_{-z}\):
\[
L^{(\nu)}(1-z+2n+2\nu)
=
\theta_{2n+1}^{(\nu)}
\proscal{P_{2n+1}^{(\nu)}}{\cKnu_{-z}},
\qquad n\ge0.
\]
The functional equation then gives the corresponding negative shifts:
\[
L^{(\nu)}\bigl(z-(2n+2\nu)\bigr)
=
\Chi^{(\nu)}(z-2n-2\nu)\,
\theta_{2n+1}^{(\nu)}
\proscal{P_{2n+1}^{(\nu)}}{\cKnu_{-z}},
\]
and
\[
L^{(\nu)}\bigl(1-z-(2n+2\nu)\bigr)
=
\Chi^{(\nu)}(1-z-2n-2\nu)\,
\theta_{2n+1}^{(\nu)}
\proscal{P_{2n+1}^{(\nu)}}{\cKnu_{z-1}},
\]
where \(n\ge0\) and \(2n+2\nu\ge1\).

\begin{table}[htb!]
\centering
\begin{tabular}{c|c|c}
value & formula & range of \(n\) \\
\hline
\(L^{(\nu)}(z+2n+2\nu)\)
&
\(\theta_{2n+1}^{(\nu)}  \proscal{P_{2n+1}^{(\nu)}}{\cKnu_{z-1}}\)
&
\(n\ge0\)
\\[2mm]
\(L^{(\nu)}(1-z+2n+2\nu)\)
&
\(\theta_{2n+1}^{(\nu)}  \proscal{P_{2n+1}^{(\nu)}}{\cKnu_{-z}}\)
&
\(n\geq0,\quad 2n+2\nu\geq0\)
\\
\(L^{(\nu)}(z-(2n+2\nu))\)
&
\(\Chi^{(\nu)}(z-2n-2\nu)\,
 \theta_{2n+1}^{(\nu)}  \proscal{P_{2n+1}^{(\nu)}}{\cKnu_{-z}}\)
& \(n\ge1\) \\
\(L^{(\nu)}(1-z-(2n+2\nu))\) & \(\Chi^{(\nu)}(1-z-2n-2\nu)\,
\theta_{2n+1}^{(\nu)} \proscal{P_{2n+1}^{(\nu)}}{\cKnu_{z-1}}\)
& \(n\geq0,\quad 2n+2\nu\geq0\) \\
\hline
\end{tabular}
\caption{Computation of the shifts generated by the odd polynomial
orders \(k=2n+1\).}
\label{tab:shift-values}
\end{table}

\item On the critical line \(\RE(z)=\tfrac12\), write
\(z=\tfrac12+ic\), \(c\in\R\). Then \(1-z=\bar z\), the kernels are
conjugate, and \(|\Chi^{(\nu)}(\tfrac12+ic)|=1\). Since both \(\eta\)
and \(\beta\) have real Dirichlet coefficients, the Schwarz reflection
principle gives
\[
L^{(\nu)}(s)=\overline{L^{(\nu)}(\bar{s})}
\]
in the half-plane of absolute convergence, and hence everywhere by
analytic continuation; see \cite[Section~2.1]{Titchmarsh1986}.
For every \(n\ge0\), we obtain
\[
\proscal{P_{2n+1}^{(0)}}{\cKz_{-1/2-ic}}
  = -\,\overline{\proscal{P_{2n+1}^{(0)}}{\cKz_{-1/2+ic}}},
\]
and
\[
\proscal{P_{2n+1}^{(1/2)}}{\cKh_{-1/2-ic}}
  = -\,\overline{\proscal{P_{2n+1}^{(1/2)}}{\cKh_{-1/2+ic}}}.
\]
The minus sign is the parity factor
\[
P_{2n+1}^{(\nu)}(-t)=-P_{2n+1}^{(\nu)}(t).
\]
Indeed,
\(\overline{\cKnu_u(\pi t)}=\cKnu_{\bar u}(-\pi t)\); the substitution
\(t\mapsto -t\) brings the reflected kernel back to \(\pi t\), at the
price of this parity factor.
\end{itemize}

\section{Concluding Remarks}
\label{sec:conclusion}

We have derived asymptotic expansions for the entries of the inverse of
Vandermonde matrices with nodes \(u_n=n^2\) and \(u_n=(n-\tfrac12)^2\).
These expansions, expressed in terms of central factorial numbers, are of
independent interest and may be relevant in random matrix theory, where
Vandermonde matrices with quadratic nodes appear naturally.  The
mechanism is not restricted to quadratic nodes.  Choosing \(u_n=n\), the
elementary symmetric polynomials \(\sigma_{l-k}^{l-1}(\set{1,\dots,l-1})\)
are the unsigned Stirling numbers of the first kind
\(\stirlingfirst{l}{k}\), and the restricted products reduce to
factorials, so the coefficients take the form
\[ w_{kn}^{[N]} = (-1)^{n+k}n
\sum_{l=\max(k,n)}^{N}\binom{l}{l-n}
\frac{1}{l!}\stirlingfirst{l}{k},
\]
where \(\stirlingfirst{l}{k}\) denotes the unsigned Stirling number of
the first kind. A geometric progression \(u_n=q^n\) replaces the
ordinary elementary symmetric polynomials by their \(q\)-analogues, and
the coefficients become finite sums of \(q\)-Stirling numbers weighted
by Gaussian binomials.

The ratio of determinants \(R_N(s;\bu)\) studied in
Section~\ref{sec:ratio-determinant} provides a second point of contact
with random matrix theory.  Ratios of determinants by Vandermonde
determinants are a standard object in random matrix theory, where they
appear in expectation values of characteristic polynomials.
In the unpublished work~\cite{iovleff2022}, we showed that the same ratio
admits a probabilistic representation as the expectation of a functional
of a random vector with Dixon--Anderson density \cite{forrester2010log}.

The shift representation developed in Sections~\ref{sec:shift-structure}
and~\ref{sec:unified-kernel} is, to our knowledge, new.
The last part of our work associating the shifts with the functional
relation can be described as the infinite dihedral group \(D_\infty\)
generated by the unit shift and the reflection \(z\mapsto 1-z\).
The functional equation supplies a cocycle that twists
the action (\cite{brown1982cohomology}, \cite{montgomery1993hopf}).
We have not pursued this algebraic interpretation, as it
lies outside our expertise, but the dihedral structure and the role
of the cocycle seem to be worth a systematic investigation.

\section*{Afterword}
A more personal closing word about the use of AI.  I have made intensive
use of it over the past months, running several open-source AIs
concurrently.  The uses ranged from literature search and reformulation
to the testing of hypotheses and the checking of computations; the
choice of what to keep, what to discard, and what to prove differently
remained mine throughout.  As a result, my way of working and of
producing the results of this article has been completely overturned.

Without going into details, every part of this article is the result of
a ceaseless back-and-forth between myself and the AIs.  I can
nevertheless state firmly that everything written here is mine.  The
text has been carefully reread; the unnecessary repetitions, the
cross-references to other parts and theorems of the text, and the
citations are my choices.  Every displayed identity has been verified
independently; the responsibility for every claim, every citation and
every decision to include or discard a result is mine.  The tools were
most useful where the algebra is long.  Several claims of the first
draft (and sometimes of the second or third) were corrected only after
a prompt to the tools and a proof worked out by hand, run concurrently,
had traced the error back to its source.  But above all, and this is the
most important point, the coherence of the whole came only from me.

Even though mathematics is presented as the discovery of new results, it
is also a human construction.  To return to the analogy of the walk,
reading an article is as much the discovery of a new landscape as the
discovery of a result.  In some sense, the way the result is presented
matters as much as the result. The path I chose is the landscape
I offer.


\begin{thebibliography}{10}

\bibitem{abramowitz1988handbook}
Milton Abramowitz, Irene~A Stegun, and Robert~H Romer.
\newblock Handbook of mathematical functions with formulas, graphs, and
  mathematical tables, 1988.

\bibitem{apostol1998introduction}
T.M. Apostol.
\newblock {\em Introduction to Analytic Number Theory}.
\newblock Undergraduate Texts in Mathematics. Springer New York, 1998.

\bibitem{biane2001probability}
Philippe Biane, Jim Pitman, and Marc Yor.
\newblock Probability laws related to the jacobi theta and riemann zeta
  functions, and brownian excursions.
\newblock {\em Bulletin of the American Mathematical Society}, 38(4):435--465,
  2001.

\bibitem{brown1982cohomology}
Kenneth~S. Brown.
\newblock {\em Cohomology of Groups}, volume~87 of {\em Graduate Texts in
  Mathematics}.
\newblock Springer-Verlag, New York, 1982.

\bibitem{Butzer1989}
P.~L. Butzer, K.~Schmidt, E.L. Stark, and L.~Vogt.
\newblock Central factorial numbers; their main properties and some
  applications.
\newblock {\em Numerical Functional Analysis and Optimization},
  10(5-6):419--488, 1989.

\bibitem{DLMF}
Nist digital library of mathematical functions.
\newblock \url{https://dlmf.nist.gov/}, 2025.
\newblock Release 1.2.4 of 2025-03-15.

\bibitem{Egorychev:1984:IRC}
G.~P. Egorychev.
\newblock {\em Integral Representation and the Computation of Combinatorial
  Sums}, volume~59 of {\em Translations of Mathematical Monographs}.
\newblock American Mathematical Society, Providence, RI, 1984.
\newblock Translated from the Russian by H. H. McFadden, Translation edited by
  Lev J. Leifman.

\bibitem{forrester2010log}
Peter~J. Forrester.
\newblock {\em Log-Gases and Random Matrices}, volume~34 of {\em London
  Mathematical Society Monographs}.
\newblock Princeton University Press, 2010.

\bibitem{gradshteyn2014}
I.~S. Gradshteyn and I.~M. Ryzhik.
\newblock {\em Table of integrals, series, and products}.
\newblock Elsevier/Academic Press, Amsterdam, eigthth edition, 2014.

\bibitem{iovleff2022}
Serge Iovleff.
\newblock A new probabilistic representation of the alternating zeta function
  and a new selberg-like integral evaluation.
\newblock arXiv preprint math/2203.09787, 2022.

\bibitem{knuth1968art}
Donald~E. Knuth.
\newblock {\em The Art of Computer Programming, Volume 1: Fundamental
  Algorithms}.
\newblock Addison-Wesley, 1968.

\bibitem{lewin1981polylogarithms}
Leonard Lewin.
\newblock {\em Polylogarithms and associated functions}.
\newblock North-Holland, Amsterdam, 1981.

\bibitem{macdonald1998symmetric}
Ian~Grant Macdonald.
\newblock {\em Symmetric functions and Hall polynomials}.
\newblock Oxford university press, 1998.

\bibitem{Macon1958}
N.~Macon and A.~Spitzbart.
\newblock Inverses of {V}andermonde matrices.
\newblock {\em The American Mathematical Monthly}, 65(2):95--100, 1958.

\bibitem{montgomery1993hopf}
Susan Montgomery.
\newblock {\em Hopf Algebras and Their Actions on Rings}, volume~43 of {\em
  CBMS Regional Conference Series in Mathematics}.
\newblock American Mathematical Society, Providence, RI, 1993.

\bibitem{spiegel2009mathematical}
Murray~R. Spiegel, Seymour Lipschutz, and John Liu.
\newblock {\em Mathematical Handbook of Formulas and Tables}.
\newblock McGraw-Hill, New York, 3rd edition, 2009.

\bibitem{Titchmarsh1986}
E.~C. Titchmarsh.
\newblock {\em The Theory of the Riemann Zeta-Function}.
\newblock Oxford University Press, 2nd edition, 1986.

\bibitem{Turner1966}
L.~Richard Turner.
\newblock Inverse of the vandermonde matrix with applications.
\newblock {NASA} Technical Note NASA TN D-3547, National Aeronautics and Space
  Administration, Washington, D.C., aug 1966.
\newblock Lewis Research Center, Cleveland, Ohio.

\end{thebibliography}

\appendix

\section{Inverse of Vandermonde Matrices}
\label{app:inverseVandermonde}

The proof of these formulas can be found at 
\href{https://proofwiki.org/wiki/Inverse_of_Vandermonde_Matrix}{Inverse of Vandermonde Matrix}
and is proposed as Exercise 40 from section 1.2.3 in \cite{knuth1968art}.
Consider the Vandermonde matrix
\begin{equation*}
V_N =
\begin{pmatrix}
  1    & u_1    & u_1^2  & \cdots & u_1^{N-1} \\
  1    & u_2    & u_2^2  & \cdots & u_2^{N-1} \\
\vdots & \vdots & \ddots & \vdots & \vdots \\
  1    &   u_N  & u_N^2  & \cdots & u_N^{N-1} \\
\end{pmatrix}
\end{equation*}
If $u_i\neq u_j$ for $i\neq j$, the determinant of $V_N$ is
$$
\Det{V_N} = \prod_{1 \mathop \le i \mathop < j \mathop \le N}
 \left({u_j - u_i}\right) \ne 0.
$$
Since this is nonzero, the inverse matrix, denoted $W_{N} = \left[{w_{ij}^{[N]}}\right]_{1\leq i,j\leq N}$, is guaranteed to exist.
By the definition of the matrix product and the identity $W_NV_N=I$, we have
$$\sum_{k = 1}^N u_i^{k-1} w_{kj}^{[N]}  = \delta_{ij}\quad
1\leq i,j\leq N.
$$
For $1\leq n \leq N$, if $P_n \left({x}\right)$ is the polynomial
$$P_n \left(x\right) := \sum_{k=1}^N w_{kn}^{[N]} x^{k-1}$$
then
$P_n \left({u_1}\right) = 0, \ldots, P_n \left({u_{n-1}}\right) = 0,
 P_n \left({u_n}\right) = 1, P_n \left({u_{n+1}}\right) = 0, \ldots,
  P_n \left({u_N}\right) = 0.$
By Lagrange's interpolation formula, the $n$th column of $W_{N}$ is composed of the
coefficients of the $n$th Lagrange basis polynomial
$$
P_n \left({x}\right) = \sum_{k = 1}^N w_{kn}^{[N]} x^{k-1}
= \prod_{\substack {1 \leq j \leq N \\ j \neq n}} \frac {x - u_j} {u_n - u_j}.
$$
Expanding the product identifies $w_{kn}^{[N]}$ as
$$
w_{kn}^{[N]} = (-1)^{N-k} \frac
{\sigma_{N-k}^{N-1}(\left\{ u_1,\ldots,u_N \right\}\backslash \left\{ u_n \right\})}
{\displaystyle\prod_{\substack {1 \leq j \leq N \\ j \neq n}} (u_n - u_j)}
$$
with $\sigma_{N-k}^{N-1}$ denoting the elementary symmetric polynomials
of degree $N-k$
in $N-1$ variables (\cite{macdonald1998symmetric}). In particular
\begin{itemize}
\item The first row of the matrix $W_{N}$ ($k=1$) is given by
$$
w_{1n}^{[N]}= (-1)^{N-1}
  \prod_{\substack {1 \leq j \leq N \\ j \neq n}}  \frac{u_j}{u_n-u_j}
= \prod_{\substack {1 \leq j \leq N \\ j \neq n}}  \frac{u_j}{u_j-u_n}
\qquad n=1,\ldots,N.
$$
\item The last row of the matrix $W_{N}$ ($k=N$) is given by
$$
w_{Nn}^{[N]}=  \prod_{\substack {1 \leq j \leq N \\ j \neq n}}  \frac{1}{u_n-u_j}
\qquad n=1,\ldots,N.
$$
\item The last column of the matrix $W_{N}$ ($n=N$) is given by
$$
w_{kN}^{[N]} = (-1)^{N-k}
\frac{\sigma_{N-k}^{N-1}(u_1,\ldots,u_{N-1})}
     {\displaystyle\prod_{\substack {1 \leq j \leq N-1}} (u_N - u_j)},
\qquad k=1,\ldots, N.
$$
\end{itemize}

\section{The unsigned central factorial numbers (cfn)}
\label{app:cfn}
The central factorial numbers arise in finite-difference calculus with
central differences and in expansions of central factorial powers. They
play a role analogous to that of the Stirling numbers, but with respect
to the central-difference operator rather than the forward-difference
operator.
The even unsigned central factorial numbers, denoted in this work by
$t_0(k,n) = |t(2n,2k)|$, (\url{https://oeis.org/A008955}),
and the odd unsigned central factorial numbers, denoted by
$t_{\frac12}(k,n) = |t(2n+1,2k+1)|$, (\url{https://oeis.org/A008956}),
are obtained from the recurrences
\begin{eqnarray*}
t_0(k,n) &=& t_0(k-1,n-1) + (n-1)^2 \, t_0(k,n-1),\\
t_{\frac12}(k,n) &=& t_{\frac12}(k-1,n-1) + \left(n-\frac{1}{2}\right)^2 \, t_{\frac12}(k,n-1),    
\end{eqnarray*}
with the boundary conditions
$t_0(n,n) = t_{\frac12}(n,n) = 1$ for all $n\geq 0$, and
$$
t_0(0,n) = 0, \qquad 
t_{\frac12}(0,n) = \frac{(2n-1)!!^2}{4^n}, \qquad n \ge 1,
$$
where $(2n+1)!! = 1 \cdot 3 \cdot 5 \cdots (2n+1)$ denotes the double factorial. 
We adopt the convention that $t_0(k,n)=t_{\frac12}(k,n) = 0$ if $k>n$.
These numbers are presented and studied in \cite{Butzer1989}. From that reference, 
we have the following generating functions for $|z| \le 2$ and $k \in \mathbb{N}$ 
(\cite[Theorem 4.1.2]{Butzer1989}):
\begin{eqnarray}
\frac{1}{(2k)!} \left[2\,\arcsin\left(\frac{z}{2}\right)\right]^{2k}
&=& \sum_{j=k}^{\infty} t_0(k,j) \, \frac{z^{2j}}{(2j)!}
\label{eq:gen-funct-cfn-0}\\
\frac{1}{(2k+1)!} \left[2\,\arcsin\left(\frac{z}{2}\right)\right]^{2k+1}
&=& \sum_{j=k}^{\infty} t_{\frac12}(k,j) \, \frac{z^{2j+1}}{(2j+1)!}
\label{eq:gen-funct-cfn-12}
\end{eqnarray}

\section{Generalized Clausen Functions}
\label{app:Clausen}

Let $z\in\C$. Throughout this appendix, we use the conventions
\begin{equation}\label{eq:Clausen-def}
\Sl_z(\theta) \eqdef \sum_{n=1}^\infty \frac{\sin (n\theta)}{n^z}
\qquad\text{and}\qquad 
\Cl_z(\theta) \eqdef \sum_{n=1}^\infty \frac{\cos (n\theta)}{n^z},
\end{equation}
which define the \emph{generalized Clausen functions}.  
These differ slightly from the notation of Lewin~\cite{lewin1981polylogarithms}, 
where the trigonometric factor depends on the parity of the index.
Our $\Sl_z$ and $\Cl_z$ are simply the two components of the polylogarithm
on the unit circle:
\begin{equation}\label{eq:Clausen-polylog}
\Li_z\!\left(e^{i\theta}\right)=\Cl_z(\theta)+i\,\Sl_z(\theta),
\qquad \theta\in(0,2\pi).
\end{equation}
The series \eqref{eq:Clausen-def} converges absolutely for $\RE(z)>1$ and 
conditionally for $\RE(z)>0$ when $\theta\not\in2\pi\Z$.
Differentiating under the sum (or using \eqref{eq:Clausen-polylog}) gives the 
recurrence relations
\begin{equation}\label{eq:Clausen-diff}
\frac{\partial}{\partial\theta}\Cl_z(\theta)=-\Sl_{z-1}(\theta),
\qquad
\frac{\partial}{\partial\theta}\Sl_z(\theta)=\Cl_{z-1}(\theta).
\end{equation}

\paragraph{Positive integer order.}
For $j\in\N^\star$ and $\theta\in[0,2\pi]$ one has the Fourier-Bernoulli 
formulas (see \cite[p.~167]{apostol1998introduction})
\begin{equation}\label{eq:Clausen-odd}
\Sl_{2j-1}(\theta) = \frac{(-1)^j}{2}\frac{(2\pi)^{2j-1}}{(2j-1)!}\,
B_{2j-1}\!\left(\frac{\theta}{2\pi}\right),
\quad\text{and}\quad
\Cl_{2j}(\theta) = \frac{(-1)^{j+1}}{2}\frac{(2\pi)^{2j}}{(2j)!}\,
B_{2j}\!\left(\frac{\theta}{2\pi}\right),
\end{equation}
where $B_n(x)$ denotes the Bernoulli polynomials.  
For $j=1$ the formula \eqref{eq:Clausen-odd} is valid for 
$\theta\in(0,2\pi)$ because of the jump discontinuity of $B_1$ at the 
endpoints; for $j\ge2$ it holds on the closed interval $[0,2\pi]$.
The Clausen function of order two has the well-known integral representation 
(\cite{abramowitz1988handbook}, Ch.~27.8)
\begin{equation}\label{eq:Clausen-2}
\Sl_2(\theta)= \sum_{n=1}^\infty \frac{\sin(n\theta)}{n^2}
= - \int_0^\theta \ln\!\left(2\sin\frac{t}{2}\right) dt,
\qquad \theta\in(0,2\pi).    
\end{equation}

\paragraph{Orders zero and one.}
In particular, differentiating \eqref{eq:Clausen-2} once yields
\[
\Cl_1(\theta)=\frac{d}{d\theta}\Sl_2(\theta)
= -\ln\!\left(2\sin\frac{\theta}{2}\right), \qquad \theta\in(0,2\pi),
\]
and differentiating once more (or using \eqref{eq:Clausen-diff} with $z=2$) 
gives
\[
\Sl_0(\theta) = -\frac{d^2}{d\theta^2}\Sl_2(\theta)
= \frac{1}{2}\cot\!\left(\frac{\theta}{2}\right), \qquad \theta\in(0,2\pi).
\]
Equivalently, $\Sl_0$ and $\Cl_0$ are obtained from the polylogarithm 
$\Li_0(e^{i\theta})=-\tfrac12+\tfrac{i}{2}\cot(\theta/2)$.

\paragraph{Special values at the endpoints.}
For $\RE(z)>0$ one has immediately \( \Sl_z(0)=\Sl_z(\pi)=0\).
Moreover, by evaluating the cosine series at $\theta=0$ and $\theta=\pi$,
\[
\Cl_{z}(0) = \zeta(z),\qquad \Cl_z(\pi)= -\eta(z),
\]
where $\zeta$ is the Riemann zeta function and $\eta(z)=(1-2^{1-z})\zeta(z)$ 
is the Dirichlet eta function.

\section{Generalized half-shifted Clausen Functions}
\label{app:Half-Clausen}

Let $z\in\C$. We define the \emph{generalized half-shifted Clausen functions} by
\begin{equation}\label{eq:Half-Clausen-def}
\Slh_z(\theta) \eqdef \sum_{n=1}^\infty 
\frac{\sin\left((n-\tfrac12)\theta\right)}{(n-\tfrac12)^z}
\qquad\text{and}\qquad
\Clh_z(\theta) \eqdef \sum_{n=1}^\infty 
\frac{\cos\left((n-\tfrac12)\theta\right)}{(n-\tfrac12)^z},
\end{equation}
which converge absolutely for $\RE(z)>1$.  These functions are related to
the half-shifted polylogarithm by the algebraic identity
\begin{equation}\label{eq:Half-Clausen-Lerch}
\Lih_z(\theta)\eqdef \sum_{n=1}^{\infty} \frac{e^{i(n-1/2)\theta}}{(n-\tfrac12)^z}
= e^{i\theta/2}\,\Phi\!\left(e^{i\theta},z,\tfrac12\right)
= \Clh_z(\theta)+i\,\Slh_z(\theta),
\end{equation}
where $\Phi(w,z,a)$ denotes the Lerch transcendent. 
Differentiating under the sum (or using \eqref{eq:Half-Clausen-Lerch}) gives the 
recurrence relations
\begin{equation}\label{eq:Half-Clausen-diff}
\frac{\partial}{\partial\theta}\Clh_z(\theta)=-\Slh_{z-1}(\theta),
\qquad
\frac{\partial}{\partial\theta}\Slh_z(\theta)=\Clh_{z-1}(\theta).
\end{equation}

\paragraph{Positive integer order.}
For $j\in\N^\star$ and $\theta\in[0,2\pi]$ the half-shifted Clausen functions
of even and odd order are expressed through the Euler polynomials $E_n(x)$ by
\begin{equation}\label{eq:Half-Clausen-even-cos}
\Clh_{2j}(\theta) = \frac{(-1)^j}{4}\,\frac{(2\pi)^{2j}}{(2j-1)!}\,
E_{2j-1}\!\left(\frac{\theta}{2\pi}\right)
\quad \text{and}\quad
\Slh_{2j+1}(\theta) = \frac{(-1)^j}{4}\,\frac{(2\pi)^{2j+1}}{(2j)!}\,
E_{2j}\!\left(\frac{\theta}{2\pi}\right).
\end{equation}
The complementary orders $\Slh_{2j}$ and $\Clh_{2j+1}$
do not reduce to elementary polynomials; the lowest nontrivial example is the
half-shifted Clausen integral
\begin{equation}\label{eq:Half-Clausen-2}
\Slh_{2}(\theta)
= \int_0^\theta \ln\!\left(\cot\frac{t}{4}\right)dt,
\qquad \theta\in(0,2\pi),
\end{equation}
which follows from termwise integration of the series for $\Clh_1$.

\paragraph{Orders zero and one.}
Differentiating \eqref{eq:Half-Clausen-2} once gives
\[
\Clh_1(\theta) = \frac{d}{d\theta}\Slh_{2}(\theta)
= \ln\!\left(\cot\frac{\theta}{4}\right), \qquad \theta\in(0,2\pi),
\]
and a second differentiation yields, via
$\partial_\theta^2\Slh_z(\theta)=-\Slh_{z-2}(\theta)$,
\[
\Slh_0(\theta) = -\frac{d^2}{d\theta^2}\Slh_2(\theta)
= \frac12\csc\!\left(\frac{\theta}{2}\right), \qquad \theta\in(0,2\pi).
\]
Equivalently, from the Lerch transcendent at $z=0$ one obtains
$\Lih_0(\theta)=i/(2\sin(\theta/2))$, whence
\begin{equation}\label{eq:Half-Clausen-0}
\Clh_0(\theta)=0, \qquad
\Slh_0(\theta)=\frac12\csc\!\left(\frac{\theta}{2}\right).
\end{equation}

\paragraph{Special values at the endpoints.}
For $\RE(z)>1$ one has immediately \(\Slh_z(0)=0\).
Evaluating the cosine series at $\theta=0$ gives
\begin{equation}\label{eq:Half-Clausen-theta0}
\Clh_z(0) = \sum_{n=1}^{\infty}\frac{1}{(n-\tfrac12)^z}= (2^z-1)\zeta(z),
\qquad z\neq1,
\end{equation}
where $\zeta$ is the Riemann zeta function.  At $\theta=\pi$ the cosine
terms vanish identically, while the sine series produces the Dirichlet beta
function $\beta(z)$:
\begin{equation}\label{eq:Half-Clausen-thetapi}
\Clh_z(\pi)=0,
\qquad
\Slh_z(\pi) = \sum_{n=1}^{\infty}\frac{(-1)^{n-1}}{(n-\tfrac12)^z}
= 2^z\beta(z).
\end{equation}
In particular, $\Slh_1(\theta)=\pi/2=2\beta(1)$ for every $\theta\in(0,2\pi)$,
so $\Slh_1(\pi)=\pi/2=2\beta(1)$.

\section{Proof of the theorem \ref{thm:continuation}}
\label{app:jonquiere}
Let \(0<t<1\). We start from Lerch's transformation formula
\cite[Entry~9.553]{gradshteyn2014}.
\begin{equation}\label{eq:Lerch-transform}
\begin{aligned}
\Phi(e^{i\pi t},z-1,a) =&i(2\pi)^{z-2}e^{-i\pi ta}\Gamma(2-z)
\\
&\times
\left[
  e^{-i\pi(z-1)/2}\Phi\left(e^{-2\pi ia},2-z,\frac t2\right)
- e^{i\pi(2a+(z-1)/2)}\Phi\left(e^{2\pi ia},2-z,1-\frac t2\right)
\right].
\end{aligned}
\end{equation}
The two kernels correspond to the two specializations \(a=1\) and
\(a=\tfrac12\).

\paragraph{Case \(\nu=0\).}
Since
\[
\Liz_{z-1}(\pi t)=e^{i\pi t}\Phi(e^{i\pi t},z-1,1),
\]
the factors \(e^{i\pi t}\) and \(e^{-i\pi t}\) cancel in
\eqref{eq:Lerch-transform}. Moreover,
\[
       e^{-i\pi(z-1)/2}=i e^{-i\pi z/2},
\qquad e^{i\pi(2+(z-1)/2)}=-i e^{i\pi z/2}.
\]
Therefore
\begin{equation}\label{eq:Li-Jonquiere}
\Liz_{z-1}(\pi t)
=
-\frac{\Gamma(2-z)}{(2\pi)^{2-z}}
\left[
e^{-i\pi z/2}\zeta\!\left(2-z,\frac t2\right)
+
e^{i\pi z/2}\zeta\!\left(2-z,1-\frac t2\right)
\right].
\end{equation}
For \(a>0\) and \(\RE(z)<1\),
\begin{equation}\label{eq:Hurwitz-integral}
\zeta(2-z,a)
=
\frac1{\Gamma(2-z)}
\int_0^\infty
\frac{v^{1-z}e^{-av}}{1-e^{-v}}\,dv.
\end{equation}
Substituting \eqref{eq:Hurwitz-integral} into
\eqref{eq:Li-Jonquiere} gives
\[
\Liz_{z-1}(\pi t)
=
-\frac1{(2\pi)^{2-z}}
\int_0^\infty
\frac{v^{1-z}}{1-e^{-v}}
\left[
e^{-tv/2-i\pi z/2}
+
e^{-(1-t/2)v+i\pi z/2}
\right]dv.
\]
Since
\[
1-e^{-v}=2e^{-v/2}\sinh(v/2),
\]
this becomes
\[
\Liz_{z-1}(\pi t)
=
-\frac1{(2\pi)^{2-z}}
\int_0^\infty
v^{1-z}
\frac{
\cosh\left((1-t)v/2-i\pi z/2\right)
}{
\sinh(v/2)
}\,dv.
\]
Setting \(v=2u\), we obtain \eqref{eq:Li-continuation}.

\paragraph{Case \(\nu=\tfrac12\).}
Since
\[
\Lih_{z-1}(\pi t) = e^{i\pi t/2}\Phi(e^{i\pi t},z-1,\tfrac12),
\]
the factors \(e^{i\pi t/2}\) and \(e^{-i\pi t/2}\) cancel in
\eqref{eq:Lerch-transform}. Now \(e^{-\pi i }=e^{\pi i}=-1,\)
and
\[
e^{-i\pi(z-1)/2}=i e^{-i\pi z/2},
\qquad e^{i\pi(1+(z-1)/2)}=i e^{i\pi z/2}.
\]
Consequently,
\begin{equation}\label{eq:Lih-Jonquiere}
\Lih_{z-1}(\pi t)=\frac{\Gamma(2-z)}{(2\pi)^{2-z}}
\left[
 e^{i\pi z/2} \Phi\!\left(-1,2-z,1-\frac t2\right)
-e^{-i\pi z/2}\Phi\!\left(-1,2-z,\frac t2\right)
\right].
\end{equation}
For \(\RE(z)<1\), the alternating periodic Lerch function satisfies
\begin{equation}\label{eq:periodic-Hurwitz}
\Phi(-1,2-z,x)=2^{-(2-z)}
\left[
 \zeta\!\left(2-z,\frac x2\right)
-\zeta\!\left(2-z,\frac{x+1}{2}\right)
\right].
\end{equation}
Applying \eqref{eq:periodic-Hurwitz} to \eqref{eq:Lih-Jonquiere}
gives
\begin{equation}\label{eq:Lih-Hurwitz}
\begin{aligned}
\Lih_{z-1}(\pi t) =& \frac{\Gamma(2-z)}{(4\pi)^{2-z}}
\Bigg\{ e^{i\pi z/2}
\left[
  \zeta\!\left(2-z,\frac{2-t}{4}\right)
- \zeta\!\left(2-z,\frac{4-t}{4}\right)
\right]\\
&\qquad
- e^{-i\pi z/2}
\left[
  \zeta\!\left(2-z,\frac t4\right)
- \zeta\!\left(2-z,\frac{2+t}{4}\right)
\right] \Bigg\}.
\end{aligned}
\end{equation}
Substituting \eqref{eq:Hurwitz-integral} into
\eqref{eq:Lih-Hurwitz}, and then setting \(v=4u\), yields
\[
\Lih_{z-1}(\pi t) = \frac1{2\pi^{2-z}}
\int_0^\infty \frac{u^{1-z}}{\sinh(2u)}
 \left[ e^{i\pi z/2} \left(e^{tu}-e^{-(2-t)u}\right)
-e^{-i\pi z/2} \left(e^{(2-t)u}-e^{-tu}\right)\right]du.
\]
The expression in brackets equals
\[
4\sinh(u)\, \sinh\left(\frac{i\pi z}{2}-(1-t)u\right),
\]
while
\[
\frac{\sinh u}{\sinh(2u)} =\frac1{2\cosh u}.
\]
This proves \eqref{eq:Lih-continuation} for \(\RE(z)<1\). Both sides
are holomorphic for \(\RE(z)<2\): the kernel \(\Lih_{z-1}(\pi t)\) is
entire in \(z\), while the integral defines a holomorphic function in
that half-plane. Analytic continuation therefore extends the identity
to \(\RE(z)<2\).

\paragraph{Separation into cosine and sine parts.}
The identities
\[
\cosh\left((1-t)u-\frac{i\pi z}{2}\right)
=
\cosh((1-t)u)\cos\left(\frac{\pi z}{2}\right)
-
i\sinh((1-t)u)\sin\left(\frac{\pi z}{2}\right)
\]
and
\[
\sinh\left(\frac{i\pi z}{2}-(1-t)u\right)
=
-\sinh((1-t)u)\cos\left(\frac{\pi z}{2}\right)
+
i\cosh((1-t)u)\sin\left(\frac{\pi z}{2}\right)
\]
hold for every complex \(z\). Identifying the even and odd parts, or
equivalently using
\[
\Cl_{z-1}(\pi t)
=
\frac{\Liz_{z-1}(\pi t)+\Liz_{z-1}(-\pi t)}2,
\qquad
\Sl_{z-1}(\pi t)
=
\frac{\Liz_{z-1}(\pi t)-\Liz_{z-1}(-\pi t)}{2i},
\]
and the corresponding identities for the half-shifted kernels, gives
\eqref{eq:Clz-continuation}, \eqref{eq:Slz-continuation},
\eqref{eq:Clh-continuation}, and~\eqref{eq:Slh-continuation}.

For \eqref{eq:Li-continuation}, the integrand is
\(O(u^{-\RE z})\) near \(u=0\) and
\(O(u^{1-\RE z}e^{-tu})\) at infinity. For
\eqref{eq:Lih-continuation}, the corresponding bounds are
\(O(u^{1-\RE z})\) near \(u=0\) and
\(O(u^{1-\RE z}e^{-tu})\) at infinity. Hence the stated absolute
convergence ranges follow immediately.

\end{document}